\documentclass[a4paper,12pt]{article}
\usepackage[cp1251]{inputenc}
\usepackage[russian]{babel}
\usepackage{amsfonts, amssymb, amsmath, amsthm, amscd}
\usepackage{cite}
\author{A.A. Vasil'eva}
\title{Widths of weighted Sobolev classes on a domain
with a peak: some limiting cases\footnote{The research was carried
out with the financial support of the Russian Foundation for Basic
Research (grants no. 13-01-00022, 12-01-00554)}}
\date{}
\begin{document}

\maketitle

\newenvironment{Biblio}{%
                  \renewcommand{\refname}{\footnotesize REFERENCES}%
                  \begin{thebibliography}{99}%
                  \itemsep -0.2em}%
                  {\end{thebibliography}}

\def\inff{\mathop{\smash\inf\vphantom\sup}}
\renewcommand{\le}{\leqslant}
\renewcommand{\ge}{\geqslant}
\newcommand{\sgn}{\mathrm {sgn}\,}
\newcommand{\inter}{\mathrm {int}\,}
\newcommand{\dist}{\mathrm {dist}}
\newcommand{\supp}{\mathrm {supp}\,}
\newcommand{\R}{\mathbb{R}}
\renewcommand{\C}{\mathbb{C}}
\newcommand{\Z}{\mathbb{Z}}
\newcommand{\N}{\mathbb{N}}
\newcommand{\Q}{\mathbb{Q}}
\theoremstyle{plain}
\newtheorem{Trm}{Theorem}
\newtheorem{trma}{Theorem}

\newtheorem{Def}{Definition}
\newtheorem{Cor}{Corollary}
\newtheorem{Lem}{Lemma}
\newtheorem{Rem}{Remark}
\newtheorem{Sta}{Proposition}
\renewcommand{\proofname}{\bf Proof}
\renewcommand{\thetrma}{\Alph{trma}}
\section{Introduction}
In this paper, order estimates for Kolmogorov, Gelfand and linear
widths of weighted Sobolev classes $W^r_{p,g}$ on a domain with a
peak in a weighted Lebesgue space $L_{q,v}$ are obtained. In
particular, it is proved that if the peak is defined by the
function $\varphi(t)=t^\sigma |\log t|^\theta$,
$r+\left(\sigma(d-1)+1\right)\left(\frac 1q-\frac 1p\right)=0$,
$g\equiv 1$, $v\equiv 1$, then this singularity may have effect on
the orders of widths. This supplements the following result of
Besov \cite{besov_peak_width}: if $\varphi(t)=t^\sigma$ and
$r+\left(\sigma(d-1)+1\right)\left(\frac 1q-\frac 1p\right)>0$,
then the orders of Kolmogorov widths are the same as for domains
with the Lipschitz boundary.

Let $d\in \N$, let $\Omega \subset \R^d$ be a
bounded domain (an open connected set), and let $g$,
$v:\Omega\rightarrow (0, \, \infty)$ be measurable functions.
For each measurable vector-valued function $\psi:\ \Omega\rightarrow
\R^m$, $\psi=(\psi_k) _{1\le k\le m}$, $p\in [1, \,
\infty]$, we put
$$
\|\psi\|_{L_p(\Omega)}= \Big\|\max _{1\le k\le m}|\psi _k |
\Big\|_p=\left(\int \limits _\Omega \max _{1\le k\le m}|\psi _k(x)
|^p\, dx\right)^{1/p}.
$$
Let $\overline{\beta}=(\beta _1, \, \dots, \, \beta _d)\in
\Z_+^d:=(\N\cup\{0\})^d$, $|\overline{\beta}| =\beta _1+
\ldots+\beta _d$. For any distribution $f$ defined on $\Omega$ we
write $\displaystyle \nabla ^r\!f=\left(\partial^{r}\!
f/\partial x^{\overline{\beta}}\right)_{|\overline{\beta}| =r}$
(here partial derivatives are taken in the sense
of distributions), and denote by $m_r$ the number of
components of the vector-valued distribution $\nabla ^r\!f$.
We set
$$
W^r_{p,g}(\Omega)=\left\{f:\ \Omega\rightarrow \R\big| \; \exists
\psi :\ \Omega\rightarrow \R^{m_r}\!:\ \| \psi \|
_{L_p(\Omega)}\le 1, \, \nabla ^r\! f=g\cdot \psi\right\}
$$
\Big(we denote the corresponding function $\psi$ by
$\displaystyle\frac{\nabla ^r\!f}{g}$\Big),
$$
\| f\|_{L_{q,v}(\Omega)}{=}\| f\|_{q,v}{=}\|
fv\|_{L_q(\Omega)},\qquad L_{q,v}(\Omega)=\left\{f:\Omega
\rightarrow \R| \; \ \| f\| _{q,v}<\infty\right\}.
$$
We call the set $W^r_{p,g}(\Omega)$ a weighted Sobolev class.
Notice that $W^r_p(\Omega)=W^r_{p,1}(\Omega)$ is the non-weighted
Sobolev class.

For properties of weighted Sobolev spaces and their
generalizations, see the books \cite{triebel, kufner, turesson,
edm_trieb_book, triebel1, edm_ev_book} and the survey paper
\cite{kudr_nik}. Mazya \cite{mazya60} obtained the necessary and
sufficient condition for the embedding of $W^1_p(\Omega)$ in
$L_q(\Omega)$ in terms of isoperimetric and capacity inequalities.
Reshetnyak \cite{resh1, resh2}, Besov \cite{besov_84} and Bojarski
\cite{b_bojarski} showed that for a John domain $\Omega$ or a
domain with a flexible cone property, the condition of continuous
embedding of $W^r_p(\Omega)$ in $L_q(\Omega)$ is the same as for a
domain with the Lipschitz boundary; in \cite{besov_86} this result
was generalized for domains with decaying flexible cone condition.
The problem on embedding of weighted Sobolev classes on domains
with the irregular boundary that has zero angles was also
intensively studied (see, e.g., \cite{besov_84, evans_har,
hajl_kosk, kilp_maly, besov1, besov2, besov3, besov4, besov5,
labutin1, labutin2, mazya_poborchii, ambr_l}).

Denote by $AC[t_0, \, t_1]$ the space of absolutely continuous
functions on an interval $[t_0, \, t_1]$.

Let $d\ge 2$. Denote by $B_a(x)$ the closed Euclidean ball of
radius $a$ in $\R^d$ centered at the point $x$. Given $x=(y, \,
z)$, $y\in \R^{d-1}$, $z\in \R$, we denote by $B_a^{d-1}(y)$ the
closed Euclidean ball of radius  $a$ in $\R^{d-1}$ centered at the
point $y$. Set $B^{d-1}=B^{d-1}_1(0)$.
\begin{Def}
\label{fca} Let $G\subset\R^d$ be a bounded domain, $a>0$, $x_*\in
G$. We say that $G \in {\bf FC}(a, \, x_*)$, if for any $x\in G$
there exists a curve $\gamma _x:[0, \, T(x)] \rightarrow G$ with
the following properties:
\begin{enumerate}
\item $\gamma _x\in AC[0, \, T(x)]$, $|\dot \gamma _x|=1$ a.e.,
\item $\gamma _x(0)=x$, $\gamma _x(T(x))=x_*$,
\item $B_{at}(\gamma _x(t))\subset G$ for any
$t\in [0, \, T(x)]$.
\end{enumerate}
We write $G\in {\bf FC}(a)$ if $G\in {\bf FC}(a, \, x_*)$
for some $x_*\in G$. If $G\in {\bf FC}(a)$ for some
$a>0$, then we say that $G$ satisfies the John condition (and call
$G$ a John domain).
\end{Def}
For a bounded domain, the John condition is equivalent to the
flexible cone condition (see the definition in \cite{besov_il1}).

Let $X$, $Y$ be sets, $f_1$, $f_2:\ X\times Y\rightarrow
\R_+$. We write $f_1(x, \, y)\underset{y}{\lesssim} f_2(x, \, y)$
(or $f_2(x, \, y)\underset{y}{\gtrsim} f_1(x, \, y)$) if, for
any $y\in Y$, there exists $c(y)>0$ such that $f_1(x, \, y)\le
c(y)f_2(x, \, y)$ for each $x\in X$; $f_1(x, \,
y)\underset{y}{\asymp} f_2(x, \, y)$ if $f_1(x, \, y)
\underset{y}{\lesssim} f_2(x, \, y)$ and $f_2(x, \,
y)\underset{y}{\lesssim} f_1(x, \, y)$.

For $G\subset \R^d$, $x_*\in G$ we denote
$$
\overline{R}_{x_*}(G)=\sup_{x\in G} \|x-x_*\|_{l_2^d}, \quad
\underline{R}_{x_*}(G)=\inf _{x\in \partial G}\|x-x_*\|_{l_2^d}
$$
(here $\partial G$ is the boundary of the set $G$).
Notice that if $G\in {\bf FC}(a, \, x_*)$, then
\begin{align}
\label{rrr} \overline{R}_{x_*}(G)\underset{a,d}{\asymp}
\underline{R}_{x_*}(G).
\end{align}

For $z\in \R$ we set $\eta_z=(0, \, \dots, \, 0, \, z)\in \R^d$.

Let $\varphi:(0, \, 1)\rightarrow (0, \, \infty)$ be
an increasing Lipschitz function such that $\lim \limits_{z\to
+0}\varphi(z)=\lim \limits _{z\to +0} \varphi'(z)=0$.
\begin{Def}
\label{fcaf} Let $a>0$, $\tau_*>0$. We write $\Omega\in {\bf
FC}_{\varphi,\tau_*}(a)$ if $\Omega =\cup _{z\in (0, \,
\tau_*]} \Omega _z$, where $\Omega_z\in {\bf FC}(a, \, \eta_z)$,
\begin{align} \label{opr2k}
\kappa_\Omega:=\sup _{z\in (0, \, 1)}\frac{\overline{R}_{\eta_z}
(\Omega_z)}{z}<1
\end{align}
and
\begin{align}
\label{opr2phi} \underline{c}\varphi(z)\le
\overline{R}_{\eta_z}(\Omega_z)\le \overline{c} \varphi(z), \quad
z\in (0, \, \tau_*]
\end{align}
for some $0<\underline{c}<\overline{c}<\infty$.
\end{Def}

In \cite{mazya_poborchii} a criterion for the continuous embedding
of $W^r_p(D_{\varphi, \, G})$ in $L_q(D_{\varphi, \, G})$ was
obtained, where
$$
D_{\varphi,\, G}=\{x=(y, \, z)\in \R^d:\, z\in (0, \, 1), \;
y/\varphi(z)\in G\}
$$
and $G\subset \R^{d-1}$ is a bounded domain satisfying the cone
condition. Notice that $D_{\varphi, \, G}\in {\bf
FC}_{\varphi,\tau_*}(a)$ for some $a>0$, $\tau_*>0$. This result
will be generalized for weighted Sobolev spaces (with weights
depending only on $z$) and domains from the class ${\bf
FC}_{\varphi,\tau_*}(a)$.

Without loss of generality we may assume that
$\tau_*=\frac 12$. Throughout this paper we denote
${\bf FC}_{\varphi}(a)={\bf FC}_{\varphi,1/2}(a)$.

Let $g_0:(0, \, \infty)\rightarrow (0, \, \infty)$, $v_0:(0, \,
\infty)\rightarrow (0, \, \infty)$ be measurable functions, $g$,
$v:\Omega \rightarrow (0, \, \infty)$, $g(y, \, z)=g_0(z)$, $v(y,
\, z)=v_0(z)$. In addition, we suppose that there is $C_*>0$
such that
\begin{align}
\label{go}
\begin{array}{c}
\frac{g_0(t)}{g_0(s)}\le C_*,\quad \frac{v_0(t)}{v_0(s)}\le C_*,
\quad z\in (0, \, 1), \\ t, \, s\in [\max\{z/2, \, z-\varphi(z)\},
\, z+\varphi(z)].
\end{array}
\end{align}
Notice that $\max\{z/2, \, z-\varphi(z)\}=z-\varphi(z)$
for sufficiently small $z$.

Let $1<p\le q<\infty$, $r\in \N$, $\delta:=r +\frac dq-\frac
dp>0$. We write $$\mathfrak{Z}_1=(p, \, q, \, r, \, d,\, a, \,
\varphi, \, \kappa_\Omega, \, \underline{c}, \, \overline{c}, \,
C_*).$$

For $x=(x_1, \, \dots, \, x_d)\in \R^d$ we set $x'=(x_1, \,
\dots, \, x_{d-1})$.

Denote $\overline{R}(z)=\overline{R}_{\eta_z} (\Omega_z)$,
$\underline{R}(z)=\underline{R}_{\eta_z} (\Omega_z)$.

For $0\le \tau_-<\tau_+\le \frac 12$ we set $\Omega_{[\tau_-, \,
\tau_+]}=\cup _{z\in [\tau_-, \, \tau_+]}\Omega_z$ (with
$\Omega_0=\varnothing$). Observe that $\Omega_{[\tau_-, \,
\tau_+]}$ is a domain.
\begin{Trm}
\label{emb_trm_peak} Let $\Omega\in {\bf FC}_\varphi(a)$, $0\le
\tau_-<\tau_+\le \frac 12$,
\begin{align}
\label{tau_mt} \tau_-<\tau_+-\underline{R}_{\tau_+}, \quad
0<\lambda<1, \quad R=\lambda\underline{R} _{\tau_+},
\end{align}
\begin{align}
\label{wwr_pg} W^r_{p,g}(\Omega_{[\tau_-, \, \tau_+]}, \,
\Gamma_0^R)=\{f\in W^r_{p,g}(\Omega_{[\tau_-, \, \tau_+]}):\,
f|_{B_R(\eta_{\tau_+})}=0\}.
\end{align}
Then the set $W^r_{p,g}(\Omega_{[\tau_-, \, \tau_+]}, \,
\Gamma_0^R)$ is bounded in $L_{q,v}(\Omega_{[\tau_-, \, \tau_+]})$
if and only if
$$
A_{[\tau_-, \, \tau_+]}:=\max\{A_{0,[\tau_-, \, \tau_+]}, \,
A_{1,[\tau_-, \, \tau_+]}\}<\infty,
$$
with
$$
A_{0,[\tau_-, \, \tau_+]}=\sup_{t\in (\tau_-, \, \tau_+)}
\left(\int \limits_{\tau_-}^t \varphi^{d-1}(z)v_0^q(z)\,
dz\right)^{1/q} \left(\int \limits _t^{\tau_+} (z-t)^{p'(r-1)}
g_0^{p'}(z) \varphi ^{\frac{d-1}{1-p}} (z) \, dz\right)^{1/p'},
$$
$$
A_{1,[\tau_-, \, \tau_+]}=\sup_{t\in (\tau_-, \, \tau_+)}
\left(\int \limits_{\tau_-}^t (t-z)^{q(r-1)} \varphi^{d-1}(z)
v_0^q(z)\, dz\right)^{1/q} \left(\int \limits _t^{\tau_+}
g_0^{p'}(z) \varphi ^{\frac{d-1}{1-p}} (z) \, dz\right)^{1/p'}.
$$
Moreover, if $I:{\rm span}\, W^r_{p,g}(\Omega_{[\tau_-, \,
\tau_+]}, \, \Gamma_0^R) \rightarrow L_{q,v}(\Omega_{[\tau_-, \,
\tau_+]})$ is the embedding operator, then $\|I\|
\underset{\mathfrak{Z}_1, \, \lambda}{\asymp} A_{[\tau_-, \,
\tau_+]}$.
\end{Trm}

Let $(X, \, \|\cdot\|_X)$ be a normed space, let $X^*$ be its
dual, and let ${\cal L}_n(X)$, $n\in \Z_+$, be the family of
subspaces of $X$ of dimension at most $n$.
Denote by $L(X, \, Y)$ the space of continuous
linear operators from $X$ into a normed
space $Y$. Also, by ${\rm rk}\, A$ denote the dimension of the
image of an operator $A\in L(X, \, Y)$, and by $\| A\|
_{X\rightarrow Y}$, its norm.

By the Kolmogorov $n$-width of a set
$M\subset X$ in the space $X$, we mean the quantity
$$d_n(M, \, X)=\inff _{L\in {\cal L}_n(X)} \sup_{x\in M}\inff_{y\in
L}\|x-y\|_X,$$ by the linear $n$-width, the quantity
$$
\lambda_n(M, \, X) =\inff_{A\in L(X, \, X), \, {\rm rk} A\le
n}\,\sup _{x\in M}\| x-Ax\| _X,
$$
and by the Gelfand $n$-width, the quantity
$$
d^n(M, \, X)=\inff _{x_1^*, \, \dots, \, x_n^*\in X^*} \sup
\{\|x\|:\; x\in M, \, x^*_j(x)=0, \; 1\le j\le n\}=
$$
$$
=\inff _{A\in L(X, \, \R^n)}\sup \{\|x\|:\; x\in M\cap \ker A\}.
$$
In \cite{pietsch1} the definition of strict $s$-numbers of a
linear continuous operator was given. In particular, Kolmogorov
numbers of an operator $A:X\rightarrow Y$ coincide with Kolmogorov
widths $d_n(A(B_X), \, Y)$ (here $B_X$ is the unit ball in the
space $X$); if the operator is compact, then its approximation
numbers coincide with linear widths $\lambda_n(A(B_X), \, Y)$ (see
the paper of Heinrich \cite{heinr}). If $X$ and $Y$ are both
uniformly convex and uniformly smooth and $A:X\rightarrow Y$ is a
bounded linear map with trivial kernel and range dense in $Y$,
then Gelfand numbers of $A$ are equal to $d^n(A(B_X), \, Y)$ (see
the paper of Edmunds and Lang \cite{edm_lang1}).

In the 1960–1970s problems concerning the values of the widths of
function classes in $L_q$ and of finite-dimensional balls $B_p^n$
in $l_q^n$ were intensively studied (see \cite{bibl6,
tikh_babaj, bib_ismag, bib_kashin, bib_majorov, bib_makovoz,
bibl9, bibl10, bibl11, bibl12, bibl13, kashin1, kulanin} and also
\cite{tikh_nvtp}, \cite{itogi_nt} and \cite{kniga_pinkusa}). Here $l_q^n$
$(1\le q\le \infty)$ is the space $\R^n$ with the norm
$$
\|(x_1, \, \dots , \, x_n)\| _q\equiv\|(x_1, \, \dots , \, x_n)\|
_{l_q^n}= \left\{
\begin{array}{l}(| x_1 | ^q+\dots+ | x_n | ^q)^{1/q},\text{ if
}q<\infty ,\\ \max \{| x_1 | , \, \dots, \, | x_n |\},\text{ if
}q=\infty ,\end{array}\right .
$$
$B_p^n$ is the unit ball in $l_p^n$.
For $p\ge q$, Pietsch \cite{pietsch1} and Stesin \cite{stesin} found
the precise values of $d_n(B_p^\nu, \, l_q^\nu)$ and
$\lambda_n(B_p^\nu, \, l_q^\nu)$.
For $p<q$, Kashin
\cite{bib_kashin}, Gluskin \cite{bib_gluskin} and Garnaev, Gluskin
\cite{garn_glus} determined the order values of widths of
finite-dimensional balls up to quantities depending only on $p$
and $q$.

Order estimates for widths of non-weighted Sobolev classes on an
interval were obtained by Tikhomirov, Ismagilov, Makovoz and
Kashin \cite{bibl6, tikh_babaj, bib_ismag, bib_kashin,
bib_makovoz}. For multidimensional cube, the upper estimate of
widths was first obtained by Birman and Solomyak \cite{birm}.
After publication of Kashin's result in \cite{bib_kashin},
estimates for widths of Sobolev classes on a $d$-dimensional torus
and their generalizations were obtained by Temlyakov and Galeev
\cite{bibl9, bibl10, bibl11, bibl12, bibl13}. Kashin
\cite{kashin1} (for $d=1$), Kulanin \cite{kulanin} and Galeev
\cite{gal1, gal2} have obtained estimates for widths of Sobolev
classes with dominating mixed smoothness in the case of
``small-order smoothness''. Here the upper estimate was not
precise for $d>1$ (involving a logarithmic factor). Order
estimates for widths of $W^r_p([0, \, 1]^d)$ in the case of
``small-order smoothness'' were obtained by DeVore, Sharpley and
Riemenschneider \cite{de_vore_sharpley}. The result of Vybiral
\cite{vybiral} on order estimates for widths of Besov classes on a
cube is also worth mentioning.

Besov in \cite{besov_peak_width} proved the result on the
coincidence of orders of widths
$$
d_n(W^r_p(K_\sigma), \, L_q(K_\sigma))\underset{p,q,r,d,\sigma}
{\asymp} d_n(W^r_p([0, \, 1]^d), \, L_q([0, \, 1]^d));
$$
here $$K_\sigma=\{(x_1, \, \dots, \, x_{d-1}, \, x_d): \; |(x_1,
\, \dots, \, x_{d-1})|^{1/\sigma}<x_d<1\},$$ $\sigma>1$,
$r-[\sigma(d-1)+1]\left(\frac 1p-\frac 1q\right)_+>0$ and some
conditions on the parameters $p$, $q$, $r$, $d$ hold (see Theorem
\ref{sob_dn} below). For $r=1$, $p=q$ and more general ridged
domains, estimates of approximation numbers were obtained by W.D.
Evans and D.J. Harris \cite{evans_har}.

The problem of estimating Kolmogorov widths of weighted Sobolev
classes and other weighted functional classes and the problem of
estimating approximation numbers of the corresponding embedding
operators was also extensively examined. The case $d=1$ was
considered by Lifshits and Linde, Edmunds, Lang, Lomakina and
Stepanov and other authors \cite{lifs_linde, edm_lang, lang_j_at1,
lom_step_hardy, step_lom}; the authors in question have found
different sufficient conditions under which orders of
approximation numbers of the embedding operator are the same as in
the non-weighted case for a finite interval. In \cite{vas_alg_an}
the weights of a special form were considered; these weights had a
singularity at a point, which affected the asymptotics of
Kolmogorov and linear widths.

An upper estimate of Kolmogorov widths of Sobolev classes on a
cube in a weighted $L_p$-space was first obtained by Birman and
Solomyak \cite{birm} (for $q>\max \{p, \, 2\}$, the orders of this
bound are not sharp). In \cite{el_kolli}, El Kolli had found the
orders of the quantities $d_n(W^r_{p,g}(\Omega), \,
L_{q,v}(\Omega))$, where $\Omega$ is a bounded domain with smooth
boundary, $p=q$, and weight functions $g$ and $v$ were equal to a
power of the distance to the boundary of $\Omega$; by using Banach
space interpolation, Triebel \cite{triebel} extended the upper
bounds to the widths $d_n(W^r_{p,g}(\Omega), \, L_{q,v}(\Omega))$
for $p\le q$. For intersections of some weighted Sobolev classes
on a cube with weights that are powers of the distance from the
boundary, order estimates of widths were obtained by Boykov
\cite{boy_1, boy_2}. In \cite{tr_jat} Triebel obtained estimates
of approximation numbers for weighted Sobolev classes with weights
that have a singularity at a point; this result was generalized in
\cite{vas_sing}. For general weights, the Kolmogorov’s and
approximation numbers of an embedding operator of Sobolev classes
in $L_p$ were estimated by Lizorkin, Otelbaev, Aitenova and
Kusainova \cite{liz_otel1, otelbaev, ait_kus1}.

It is worth noting recent results on estimates of approximation
and entropy numbers of embedding operators of Besov and
Triebel--Lizorkin classes (see, e.g., \cite{haroske1, haroske3,
vas_bes}).

Suppose that for any $0<z\le \frac 12$
\begin{align}
\label{g0v0} g_0(z)=z^{-\beta_g} |\log z|^{-\alpha_g}\rho_g(|\log
z|), \quad v_0(z)=z^{-\beta_v}|\log z|^{-\alpha_v}\rho_v(|\log
z|),
\end{align}
\begin{align}
\label{phi_def} \varphi(z)=z^\sigma |\log z|^{\theta}\omega(|\log
z|),
\end{align}
where $\rho_g$, $\rho_v$, $\omega$ are absolutely continuous functions
such that
\begin{align}
\label{limtrho} \lim \limits_{t\to+\infty}
\frac{t\rho_g'(t)}{\rho_g(t)} = \lim \limits_{t\to+\infty}
\frac{t\rho_v'(t)}{\rho_v(t)} = \lim \limits_{t\to+\infty}
\frac{t\omega'(t)}{\omega(t)}=0,
\end{align}
\begin{align}
\label{sigma_prop} \sigma>1, \quad r+(\sigma(d-1)+1)\left(\frac
1q-\frac 1p\right)=\beta_g+\beta_v, \quad \sigma(d-1)+1 -\beta_v
q>0,
\end{align}
\begin{align}
\label{alpha_prop}
\alpha:=\alpha_g+\alpha_v+\theta(d-1)\left(\frac 1p-\frac
1q\right)>0.
\end{align}
For $z>\frac 12$, we extend the functions $g_0$, $v_0$
and $\varphi$ by a constant.

We set
\begin{align}
\label{rho_def}
\rho(s):=\rho_g(s)\rho_v(s)[\omega(s)]^{(d-1)\left(\frac 1q-\frac
1p\right)},
\end{align}
$\mathfrak{Z}=(\mathfrak{Z}_1, \, g, \, v)$. Observe that $\lim
\limits _{t\to+\infty} \frac{t\rho'(t)}{\rho(t)}=0$ and the function
$t^{-\alpha}\rho(t)$ is decreasing for large $t>0$.

\label{vartheta_n_def}We set, respectively,
$\vartheta_l(M, \, X)=d_l(M, \, X)$ and $\hat q=q$,
$\vartheta_l(M, \, X)=\lambda_l(M, \, X)$ and $\hat q=\min\{q, \, p'\}$,
$\vartheta_l(M, \, X)=d^l(M, \, X)$ and $\hat q=p'$
in estimating Kolmogorov, linear and Gelfand
widths, respectively.

\begin{Trm}
\label{width_est}
\begin{enumerate}
\item Let $p=q$ or $p<q$, $\hat q\le 2$, $\alpha\ne \frac{\delta}{d}$.
Then
$$
\vartheta_n(W^r_{p,g}(\Omega), \, L_{q,v}(\Omega))
\underset{\mathfrak{Z}}{\asymp} n^{-\min \left\{\frac{\delta}{d},
\, \alpha\right\}}.
$$
\item Let $p<q$ and $\hat q>2$. We set $\theta_1=\frac{\delta}{d} +
\min \left\{ \frac 12-\frac{1}{\hat q}, \, \frac 1p-\frac
1q\right\}$, $\theta_2=\frac{\hat q\delta}{2d}$, $\theta_3=\alpha+
\min \left\{ \frac 12-\frac{1}{\hat q}, \, \frac 1p-\frac
1q\right\}$, $\theta_4 =\frac{\hat q\alpha}{2}$,
$\sigma_1=\sigma_2=0$, $\sigma_3=1$, $\sigma_4=\frac{\hat q}{2}$.
Suppose that there exists $j_*\in \{1, \, 2, \, 3, \, 4\}$ such
that $\theta_{j_*}<\min _{j\ne j_*} \theta_j$. Then
$$
\vartheta_n(W^r_{p,g}(\Omega), \, L_{q,v}(\Omega))
\underset{\mathfrak{Z}} {\asymp} n^{-\theta_{j_*}}
\rho(n^{\sigma_{j_*}}).
$$
\end{enumerate}
\end{Trm}

\section{Preliminary results}
Let $t_0<t_1$, $r>0$, and let $u$, $w:[t_0, \, t_1]\rightarrow
\R_+$ be measurable functions. Set
$$
\tilde I_{r,u,w,t_1}f(t)=w(t)\int \limits
_t^{t_1}(t-s)^{r-1}u(s)f(s)\, ds.
$$

The criterion of continuity for the operator $\tilde
I_{r,u,w,t_1}:L_p[t_0, \, t_1]\rightarrow L_q[t_0, \, t_1]$ is
proved by V.D. Stepanov \cite{stepanov2}. Let us formulate this
result for the case $p\le q$.
\begin{trma}
\label{step_trm} Let $r\ge 1$, $1<p\le q<\infty$. Then $\| \tilde
I_{r,u,w,t_1}\| _{L_p\rightarrow L_q}\underset{p,q,r}{\asymp}
B_0+B_1$, where
$$
B_0=\sup _{t\in (t_0, \, t_1)}\left(\int \limits _{t_0}^t
(t-x)^{q(r-1)} w^q(x)\, dx\right)^{1/q}\left(\int \limits _t^{t_1}
u^{p'}(x)\, dx\right)^{1/p'},
$$
$$
B_1=\sup _{t\in (t_0, \, t_1)}\left(\int \limits _{t_0}^t w^q(x)\,
dx \right)^{1/q}\left(\int \limits _t^{t_1}
(x-t)^{p'(r-1)}u^{p'}(x)\, dx\right) ^{1/p'}.
$$
\end{trma}

Denote by ${\rm mes}\, \Omega$ the Lebesgue measure of a set
$\Omega\subset \R^d$.

Reshetnyak \cite{resh1, resh2} constructed the integral
representation for smooth functions defined on a John domain
$\Omega$ in terms of their derivatives of order $r$. This together
with the result of Sobolev and Adams \cite{sobol38, adams, adams1}
implies the following theorem.
\begin{trma}
\label{resh_teor} Let $\Omega\in {\bf FC}(a)$, $r\in \N$,
$1<p<q<\infty$, $\frac rd+\frac 1q-\frac 1p\ge 0$. Then for any
function $f\in W^r_p(\Omega)$ there exists a polynomial $P_f$ of
degree not exceeding $r-1$ such that
$$
\|f-P_f\|_{L_q(\Omega)}\underset{p,q,r,d,a}{\lesssim}({\rm mes}\,
\Omega)^{\frac{\delta}{d}} \|\nabla^r f\| _{L_p(\Omega)};
$$
in addition,
$$
\|f\|_{L_q(\Omega)}\underset{p,q,r,d,a}{\lesssim} ({\rm mes}\,
\Omega)^{\frac{\delta}{d}} \|\nabla^r f\| _{L_p(\Omega)} +({\rm
mes}\, \Omega)^{\frac 1q-\frac 1p}\|f\|_{L_p(\Omega)}.
$$
\end{trma}

Kashin and Gluskin \cite{bib_kashin, bib_gluskin} obtained order
estimates for $d_n(B_p^\nu, \, l_q^\nu)$, $d^n(B_p^\nu, \,
l_q^\nu)$ and $\lambda_n(B_p^\nu, \, l_q^\nu)$.
\begin{trma} Let $1<p< q<\infty$. Then
\begin{align}
\label{gluskin} d_n(B_p^\nu, \, l_q^\nu)\underset{q,p}{\asymp}
\Phi(n, \, \nu, \, p, \, q),
\end{align}
\begin{align}
\label{gluskin_lin} \lambda _n(B_p^\nu, \,
l_q^\nu)\underset{q,p}{\asymp} \Psi(n, \, \nu, \, p, \, q),
\end{align}
\begin{align}
\label{gluskin_gelf} d^n(B_p^\nu, \,
l_q^\nu)\underset{q,p}{\asymp} \Phi(n, \, \nu, \, q', \, p'),
\end{align}
with
$$\Phi(n, \, \nu, \, p, \, q)=\left\{
\begin{array}{l} \min\bigl\{ 1, \, \left(\nu^{1/q}n^{-1/2}\right)
^{\left(\frac1p-\frac1q\right)/\left(\frac12-\frac1q\right)}\bigr\},
\;
2\le p< q< \infty, \\
\max\bigl\{\nu ^{\frac 1q-\frac 1p}, \, \min \bigl(1, \, \nu
^{\frac 1q}n^{-\frac 12}
\bigr)\bigl(1-\frac{n}{\nu}\bigr)^{1/2}\bigr\}, \; 1< p< 2< q< \infty , \\
\max\bigl\{\nu ^{\frac1q-\frac1p}, \,
\bigl(1-\frac{n}{\nu}\bigr)^{\left(\frac1q-
\frac1p\right)/\left(1-\frac{2}{p}\right)}\bigl\}, \; 1< p< q\le
2,\end{array}\right.$$
$$\Psi(n, \, \nu, \, p, \, q)=\left\{ \begin{array}{l} \Phi(n, \, \nu, \, p, \,
q),\quad
\text{ if }\quad q\le p', \\
\Phi(n, \, \nu, \, q', \, p'),\quad\text{ if }\quad
p'<q.\end{array}\right.
$$
\end{trma}

Let us formulate the result on estimates of widths $d_n(W^r_p([0,
\, 1]^d), \, L_q([0, \, 1]^d))$, $\lambda_n(W^r_p([0, \, 1]^d), \,
L_q([0, \, 1]^d))$ and $d^n(W^r_p([0, \, 1]^d), \, L_q([0, \,
1]^d))$ (see \cite{bibl6, tikh_babaj, bib_majorov, bib_ismag,
bib_kashin, bibl11, bibl13, vybiral, de_vore_sharpley}).
\begin{trma}
\label{sob_dn} Let $r\in \N$, $1\le p, \, q\le \infty$,
$\displaystyle \frac rd +\frac 1q-\frac 1p>0$. Set
$$
\theta _{p,q,r,d}=\left\{\begin{array}{l}
\frac{\delta}{d},\quad \quad\quad\mbox{ if }\ p\ge q\quad \text{or}\quad p<q, \; \hat q\le 2, \\
\min \bigl\{\frac{\delta}{d}+\min\bigl\{\frac 12-\frac{1}{\hat q},
\, \frac 1p-\frac 1q\bigr\}, \, \frac{\hat q\delta}{2d}\bigr\},
\;\; \mbox{ if }\ p<q, \; \hat q>2.
\end{array}\right.
$$
In addition, suppose that $\frac{\delta}{d}+\min\bigl\{\frac
12-\frac{1}{\hat q}, \, \frac 1p-\frac 1q\bigr\}\ne\frac{\hat
q\delta}{2d}$ in the case $p<q$, $\hat q>2$. Then
$$
\vartheta_n(W^r_p([0, \, 1]^d), \, L_q([0, \, 1]^d))
\underset{r,d,p,q}{\asymp}n^{-\theta_{p,q,r,d}}.
$$
\end{trma}
The following lemma was proved in \cite{vas_bes}.
\begin{Lem}
\label{sum_lem} Let $\Lambda_*:(0, \, +\infty) \rightarrow (0, \,
+\infty)$ be an absolutely continuous function such that $\lim
\limits_{y\to +\infty}\frac{y\Lambda _*'(y)} {\Lambda _*(y)}=0$.
Then for any $\varepsilon >0$
\begin{align}
\label{rho_yy} t^{-\varepsilon}
\underset{\varepsilon,\Lambda_*}{\lesssim}
\frac{\Lambda_*(ty)}{\Lambda_*(y)}\underset{\varepsilon,
\Lambda_*}{\lesssim} t^\varepsilon,\quad 1\le y<\infty, \;\; 1\le
t<\infty.
\end{align}
\end{Lem}

\section{The embedding theorem for weighted Sobolev classes on a domain with a peak}
\begin{Lem}
\label{un_fca} Let $0<\tau_0<\tau_1\le \frac 12$ and $c>0$ be such
that
\begin{align}
\label{t1t0} \tau_1-\tau_0\le c\varphi(\tau_0),
\end{align}
and let $L$ be a Lipschitz constant of the function
$\varphi|_{[\tau_0, \, \tau_1]}$. Then $\cup _{z\in [\tau_0, \,
\tau_1]} \Omega_z\in {\bf FC}(b)$, with $b=b(a, \, d, \, c, \, L,
\, \underline{c}, \, \overline{c})>0$.
\end{Lem}
\begin{proof}
For any $t\in [\tau_0, \, \tau_1]$ $$\varphi(\tau_0)\le
\varphi(t)\le \varphi(\tau_0)+L(t-\tau_0)\le (L\cdot
c+1)\varphi(\tau_0).$$ Therefore,
\begin{align}
\label{rt} \underline{R}_t \stackrel{(\ref{rrr})}
{\underset{a,d}{\asymp}} \overline{R}_t \stackrel
{(\ref{opr2phi})}{\underset{\underline{c}, \, \overline{c}, \, L,
\, c}{\asymp}} \varphi(\tau_0).
\end{align}
Let $z\in [\tau_0, \, \tau_1]$, $x\in \Omega_z$, and let
$\gamma_x:[0, \, T(x)]\rightarrow \Omega_z$ be the curve from
Definition \ref{fca}, $\gamma_x(T(x))=\eta_z$. Then $T(x)\le
a^{-1}\underline{R}_z$. Extend the curve $\gamma_x$ by connecting
$\eta_z$ and $\eta_{\tau_0}$ by a segment. It remains to apply
(\ref{rt}) and (\ref{t1t0}).
\end{proof}
Set $G_z=\{y\in \R^{d-1}:\, (y, \, z)\in \Omega\}$. Then
$$G_z\supset \{y\in \R^{d-1}:\;\; (y, \, z)\in \Omega_z\} \supset
B^{d-1}_{\underline{R}(z)}(0),$$ which implies
\begin{align}
\label{dgz} {\rm dist}\, (0, \, \partial G_z)\ge \underline{R}(z),
\quad {\rm diam}\, G_z\ge 2\underline{R}(z).
\end{align}
From (\ref{rrr}) and Definition \ref{fcaf} it follows that
\begin{align}
\label{rz} \underline{R}(z) \underset{a,d}{\asymp}
\overline{R}(z)\ge \underline{c} \varphi(z).
\end{align}

\begin{Lem}
\label{diam_g_z} The following order equalities hold:
\begin{align}
\label{dist} {\rm dist}\, (0, \, \partial G_z) \underset
{\mathfrak{Z}_1}{\asymp} \varphi(z), \quad {\rm diam}\, G_z
\underset{\mathfrak{Z}_1}{\asymp} \varphi(z).
\end{align}
\end{Lem}
\begin{proof}
The lower estimates follow from (\ref{dgz}) and (\ref{rz}).

Prove the upper estimates. Let
\begin{align}
\label{z01} \zeta \in (0, \, 1),\quad \Omega_\zeta\cap\{(y, \,
z):\; y\in \R^{d-1}\}\ne \varnothing.
\end{align}
It is sufficient to check that $\varphi
(\zeta)\underset{\mathfrak{Z}_1}{\asymp} \varphi(z)$ and to apply
Definition \ref{fcaf}. From (\ref{opr2k}) and (\ref{opr2phi}) it
follows that $\overline{R}(\zeta)\le \overline{c}\varphi(\zeta)$
and $\overline{R}(\zeta)\le \kappa_\Omega\zeta$. Further, $|\zeta
-z|\le \overline{R}(\zeta)$. Thus,
$\zeta-\overline{c}\varphi(\zeta)\le z\le \zeta
+\overline{c}\varphi(\zeta)$ and $(1-\kappa_\Omega) \zeta\le z\le
(1+\kappa_\Omega) \zeta$. Hence,
$z=\zeta+\theta(\zeta)\varphi(\zeta)$, with $\theta(\zeta)\in
[-\overline{c}, \, \overline{c}]$. Since the function $\varphi$ is
Lipschitz, we have
\begin{align}
\label{phi_z}
\varphi(z)=\varphi(\zeta+\theta(\zeta)\varphi(\zeta)) =
\varphi(\zeta) +\int \limits _0^{\theta(\zeta)\varphi(\zeta)}
\varphi'(\zeta+s)\, ds.
\end{align}
Recall that $\varphi'(t)\underset{t\to +0}{\to} 0$. Therefore, for
any $\varepsilon >0$ there exists $t_*(\varepsilon)>0$ such that
$\varphi'(t)<\varepsilon$ for any $t\in (0, \, t_*(\varepsilon))$.
Since $\zeta\le \frac{z}{1-\kappa_\Omega}$,
$\varphi(t)\underset{t\to +0}{\to} 0$ and
$\varphi'(t)\underset{t\to +0}{\to} 0$, there exists
$z_*(\varepsilon)>0$ such that $\varphi'(\zeta+s)<\varepsilon$ for
$s\in [-\overline{c}\varphi(\zeta), \, \overline{c} \varphi
(\zeta)]$, $z\in (0, \, z_*(\varepsilon)]$ (indeed,
$\zeta-\overline{c}\varphi(\zeta)>0$ for small $z$ and $\zeta+s\le
\frac{z}{1-\kappa_\Omega}+\overline{c}\varphi\left(\frac{z}{1-\kappa_\Omega}\right)
\underset{z\to +0}{\to} 0$). Hence, by (\ref{phi_z}), $|\varphi(z)
-\varphi(\zeta)|\le \overline{c} \varepsilon\varphi(\zeta)$. Take
$\varepsilon= \frac{1} {2\overline{c}}$ and get
$\frac{\varphi(\zeta)}{2}\le \varphi(z)\le \frac{3\varphi
(\zeta)}{2}$ for any $z\in (0, \, z_*(1/2\overline{c})]$ and
$\zeta$ from (\ref{z01}). For $z>z_*(1/2\overline{c})$ we apply
the relation $z\underset{\kappa_\Omega}{\asymp} \zeta$ and take
into account that the function $\varphi$ is non-decreasing and
Lipschitz.
\end{proof}

\begin{Lem}
\label{3_prop} Suppose that the functions $g_0$, $v_0$ satisfy
(\ref{go}), and let $\lambda_*\ge 1$. Then
$g_0(z)\underset{\mathfrak{Z}_1, \, \lambda_*}{\asymp} g_0(t)$,
$v_0(z)\underset{\mathfrak{Z}_1, \, \lambda_*}{\asymp} v_0(t)$ for
any $z\in \left(0, \, \frac 12\right]$ and for any $$t\in
\left[\max\left\{\frac{z}{2\lambda_*}, \,
z-\lambda_*\varphi(z)\right\}, \, \min
\left\{z+\lambda_*\varphi(z), \, \frac 12\right\}\right].$$
\end{Lem}
\begin{proof}
First prove that for $\lambda_*>0$
\begin{align}
\label{_t_z_z_phi_} g_0(z) \underset{\mathfrak{Z}_1, \,
\lambda_*}{\asymp} g_0(t), \quad v_0(z) \underset{\mathfrak{Z}_1,
\, \lambda_*}{\asymp} v_0(t), \quad 0<z\le \frac 12, \;\; z\le
t\le \min \left\{z+\lambda _*\varphi(z), \, \frac 12\right\}.
\end{align}
To this end we construct a sequence $\{z_{(k)}\}$ by induction.
Set $z_{(0)}=z$. Suppose that $z_{(k)}\le \frac 12$ is constructed
for some $k\in \Z_+$. If $z_{(k)}=\frac 12$, then the construction
is interrupted. If $z_{(k)}<\frac 12$, then we set $z_{(k+1)}=\min
\left\{z_{(k)}+\varphi(z_{(k)}), \, \frac 12 \right\}$. Let $k\in
\N$, $z_{(k)}<\frac 12$. Since the function $\varphi$ is
non-decreasing, we get $z_{(k)}\ge z_{(k-1)} +\varphi(z)$. Hence,
$z_{(k)}\ge z+k\varphi(z)$, and for $k=\lceil \lambda_*\rceil$ we
obtain $z_{(k)}\ge z+\lambda_* \varphi(z)$. It remains to take
into account that $g_0(t) \underset {\mathfrak{Z}_1}{\asymp}
g_0(z_{(j-1)})$, $v_0(t) \underset {\mathfrak{Z}_1}{\asymp}
v_0(z_{(j-1)})$ for any $t\in [z_{(j-1)}, \, z_{(j)}]$.

Let, now, $\max \left\{\frac{z}{2\lambda_*}, \,
z-\lambda_*\varphi(z)\right\} \le t\le z$. Since $\lim \limits
_{s\to +0}\varphi'(s)=0$, there exists $z_*=z_*(\mathfrak{Z}_1, \,
\lambda_*)$ such that $\frac z2\le z-\lambda_*\varphi(z)$ and
$|\varphi'(z)|\le \frac{1}{2\lambda_*}$ for any $z\le z_*$. Then
for each $z\le z_*$ we get $\varphi(z-\lambda_*\varphi(z))\ge
\frac{\varphi(z)}{2}$. Therefore, $z\le z-\lambda_*\varphi(z)+
2\lambda_* \varphi(z-\lambda_*\varphi(z))$. It remains to apply
(\ref{_t_z_z_phi_}). Let $z>z_*$. Prove that $g_0(z)
\underset{\mathfrak{Z}_1, \, \lambda_*}{\asymp} g_0(t)$, $v_0(z)
\underset{\mathfrak{Z}_1, \, \lambda_*}{\asymp} v_0(t)$ for any
$t\in \left[\frac{z}{2\lambda_*}, \, z\right]$. We have
$\varphi\left(\frac{z}{2\lambda_*}\right)\ge
\varphi\left(\frac{z_*}
{2\lambda_*}\right)\underset{\mathfrak{Z}_1, \lambda_*}{\asymp}
1$. Consequently, there is $c_{\mathfrak{Z}_1, \lambda_*}>0$ such
that $z\le \frac{z}{2\lambda_*}+c_{\mathfrak{Z}_1, \lambda_*}
\varphi\left( \frac{z}{2\lambda_*}\right)$. Apply
(\ref{_t_z_z_phi_}) once again and obtain the desired estimate.
\end{proof}

We say that sets $A$, $B\subset \R ^d$ do not overlap if $A\cap B$
is a Lebesgue nullset.

Let $E$, $E_1, \, \dots, \, E_m\subset \R^d$ be measurable sets.
We say that $\{E_i\}_{i=1}^m$ is a covering of $E$ if the set
$E\backslash \left(\cup _{i=1}^m E_i\right)$ is a Lebesgue
nullset. We say that $\{E_i\}_{i=1}^m$ is a partition of $E$ if
the sets $E_i$ do not overlap pairwise and the set $\left(\cup
_{i=1}^m E_i\right)\bigtriangleup E$ is a Lebesgue nullset.

Let $T$ be a covering of a set $G$. Put
$$
{\cal N}_{T,E}={\rm card}\{E'\in T:\, {\rm mes}(E\cap E')>0\},
\quad E\in T,
$$
$$
{\cal N}_T=\sup _{E\in T}{\cal N}_{T,E}.
$$

\begin{Lem}
\label{kratn_pokr} Let $0\le \tau_k<\tau_{k-1}<\dots <\tau_0\le
\frac 12$, $0<\hat c<1$,
\begin{align}
\label{tj1} \tau_{j-1}-\tau_j\ge \hat c\varphi(\tau_{j-1}),
\end{align}
$G_{(j)}=\Omega_{[\tau_j, \, \tau_{j-1}]}$, $1\le j\le k$,
$T=\{G_{(j)}\}_{j=1}^{k}$. Then ${\rm card}\, {\cal N}_T
\underset{\mathfrak{Z}_1,\hat c}{\lesssim} 1$.
\end{Lem}
\begin{proof}
Since the function $\varphi$ is Lipschitz and $\lim \limits _{z\to
0} \varphi'(z)=0$, there exists $z_0=z_0(\mathfrak{Z}_1)\in
\bigl(0, \, \frac 12 \bigr]$ such that for any $z\in (0, \, z_0]$
\begin{align}
\label{phi} \varphi(z-2\overline{c}\varphi(z))\ge
\frac{\varphi(z)}{2},\quad \varphi(z +2 \overline{c}\varphi(z))\le
2\varphi(z),\quad |\varphi'(z)|\le \frac {1}{4\overline{c}}.
\end{align}

Let us estimate
$$
{\rm card}\, \{i\in \overline{1, \,
k}\backslash\{j\}:\; G_{(i)}\cap G_{(j)}\ne \varnothing\}, \quad
j\in \overline{1, \, k}.
$$
Since the sets $G_{(i)}$ are open, the condition $G_{(i)}\cap
G_{(j)}\ne \varnothing$ is equivalent to the inequality ${\rm
mes}\, \left(G_{(i)}\cap G_{(j)}\right)>0$. Notice that if
$\tau_{i-1}\ge z_0$, then $\tau_{i-1}-\tau_i\ge \hat c
\varphi(\tau_{i-1}) \underset{\mathfrak{Z}_1,\hat c}{\asymp} 1$.
Hence,
$${\rm card}\{i\in \overline{1, \, k}:\tau_{i-1}\ge
z_0\}\underset{ \mathfrak{Z}_1, \hat c}{\lesssim} 1.
$$
Therefore, it is sufficient to estimate
$$
{\rm card}\, \{i\in \overline{1, \, k}\backslash\{j\}:\;
G_{(i)}\cap G_{(j)}\ne \varnothing, \; \tau_{i-1}<z_0\}.
$$
If $G_{(i)}\cap G_{(j)}\ne \varnothing$, then there exist $z\in
[\tau_j, \, \tau_{j-1}]$, $t\in [\tau_i, \, \tau_{i-1}]$ such that
\begin{align}
\label{brz_brt} B_{\overline{R}(z)}(\eta_z)\cap
B_{\overline{R}(t)}(\eta_t)\ne \varnothing.
\end{align}

Let $i<j$. Then from (\ref{brz_brt}) it follows that
$z+\overline{R}(z)\ge t-\overline{R}(t)$. Therefore,
\begin{align}
\label{ttj1_tz_rz_rt} t-\tau_{j-1}\le t-z\le
\overline{R}(z)+\overline{R}(t) \stackrel{(\ref{opr2phi})}{\le}
\overline{c} \varphi(z)+ \overline{c}\varphi(t) \le 2\overline{c}
\varphi(t),
\end{align}
i.e., $t-2\overline{c} \varphi(t) \le \tau_{j-1}$. By the
monotonicity of the function $\varphi$ and the inequality $t\le
\tau_{i-1} \le z_0$, we have $\frac{\varphi(t)}{2}
\stackrel{(\ref{phi})}{\le} \varphi(t-2\overline{c}\varphi(t)) \le
\varphi(\tau_{j-1})$. Applying (\ref{ttj1_tz_rz_rt}) once again,
we get $\tau_i-\tau_{j-1}\le 2\overline{c} \varphi(t)\le
4\overline{c} \varphi(\tau_{j-1})$. On the other hand,
$\tau_{i-1}-\tau_i\stackrel{(\ref{tj1})}{\ge} \hat c \varphi
(\tau_{i-1})\ge \hat c \varphi(\tau_{j-1})$. This yields the
desired estimate.

Let, now, $i>j$. Then from (\ref{brz_brt}) it follows that
\begin{align}
\label{brz_brt1} t+\overline{R}(t)\ge z-\overline{R}(z).
\end{align}
Let $z\ge z_0$. Then
$$
(1+\kappa _\Omega)\tau_{i-1}\ge (1+\kappa _\Omega)t
\stackrel{(\ref{opr2k})}{\ge} t+ \overline{R}
(t)\stackrel{(\ref{brz_brt1})}{\ge}
z-\overline{R}(z)\stackrel{(\ref{opr2k})}{\ge} (1-\kappa _\Omega)
z\ge (1-\kappa _\Omega)z_0.
$$
Hence, $\tau_{i-1} \underset{\mathfrak{Z}_1} {\gtrsim} 1$,
$\varphi(\tau_{i-1})\underset{\mathfrak{Z}_1} {\gtrsim} 1$,
$\tau_{i-1}-\tau_i
\stackrel{(\ref{tj1})}{\underset{\mathfrak{Z}_1,\hat c}{\asymp}}
1$ and
$$
{\rm card}\{i>j:\, G_{(i)}\cap G_{(j)}\ne \varnothing\}\underset{
\mathfrak{Z}_1, \hat c}{\lesssim} 1.
$$

Let $z<z_0$. From (\ref{brz_brt1}) it follows that
\begin{align}
\label{zrz} z-\overline{R}(z)\le
t+\overline{R}(t)\stackrel{(\ref{opr2phi})}{\le} t+
\overline{c}\varphi(t)\le \tau_{i-1}+ \overline{c} \varphi(z).
\end{align}
Set
\begin{align}
\label{til_z} \tilde z=\left\{ \begin{array}{l} z, \quad \mbox{if}
\quad z-\overline{R}(z)\le\tau_j-\overline{R}({\tau_j}), \\
\tau_j, \quad \mbox{otherwise}. \end{array}\right.
\end{align}
Then
\begin{align}
\label{zz_til_z} \tilde z-\overline{R}(\tilde z) \le \tau_{i-1}
+\overline{c} \varphi(\tilde z).
\end{align}
Indeed, if $\tilde z=z$, then it follows from (\ref{zrz}). If
$\tilde z=\tau_j$, then by (\ref{zrz}), (\ref{til_z}) and the
inequalities $t\le \tau_{i-1}\le \tau_j$ we get $$\tilde
z-\overline{R}(\tilde z)\le z-\overline{R}(z) \le t+
\overline{c}\varphi(t)\le \tau_{i-1}+ \overline{c}
\varphi(\tau_j)=\tau_{i-1} +\overline{c} \varphi(\tilde z).$$

Estimate $\tilde z-\tau_j$ from above. Taking into account the
condition $\tilde z\le z<z_0$, we get
$$
\tilde z-\tau_j\le \overline{R}(\tilde z)- \overline{R}({\tau_j})
\stackrel{(\ref{opr2phi})} {\le} \overline{c} \varphi(\tilde
z)-\underline{c}\varphi(\tau_j)=$$$$=(\overline{c}-\underline{c})
\varphi(\tau_j)+\overline{c}\int \limits_{\tau_j}^{\tilde z}
\varphi'(s)\, ds \stackrel{(\ref{phi})}{\le}
(\overline{c}-\underline{c}) \varphi(\tau_j)+\frac{\tilde
z-\tau_j}{4},
$$
which implies
\begin{align}
\label{ztj} \tilde z-\tau_j \le 2(\overline{c}-\underline{c})
\varphi(\tau_j).
\end{align}

From (\ref{zz_til_z}) we obtain that
\begin{align}
\label{zti1_} \tilde z-\tau_{i-1}\le \overline{R}(\tilde
z)+\overline{c}\varphi(\tilde z) \stackrel{(\ref{opr2phi})}{\le}
2\overline{c}\varphi(\tilde z).
\end{align}
Therefore, $\varphi(\tau_{i-1})\ge \varphi(\tilde
z-2\overline{c}\varphi(\tilde z)) \stackrel{(\ref{phi})}{\ge}
\frac{\varphi(\tilde z)}{2}$. Consequently,
$$\tau_{i-1}-\tau_i\stackrel{(\ref{tj1})}{\ge} \hat c\varphi(\tau_{i-1})\ge \frac{\hat
c}{2}\varphi(\tilde z) \ge \frac{\hat c}{2} \varphi(\tau_j).$$ On
the other hand, $\tau_j-\tau_{i-1}\le \tilde z-\tau_{i-1}
\stackrel{(\ref{zti1_})} {\le} 2 \overline{c} \varphi(\tilde z)
\stackrel{ (\ref{phi}),(\ref{ztj})}{\le} 4\overline{c}
\varphi(\tau_j)$. This yields the desired estimate.
\end{proof}

\renewcommand{\proofname}{\bf Proof of Theorem \ref{emb_trm_peak}}
\begin{proof}
The arguments are almost the same as in \cite{mazya_poborchii}.
Here we give the sketch of the proof.

In order to obtain the lower estimate, we take functions
$$\psi_f(y, \, z)=\int \limits _z^{\tau_+-R}
(t-z)^{r-1}g_0(t)f(t)\, dt,$$ where $f$ is such that
$\|f_*\|_{L_p(\Omega)}=1$ for $f_*(y, \, z)=f(z)$. By Theorem
\ref{step_trm},
$$
\sup _{\|f_*\|_{L_p(\Omega)}=1} \|\psi_f\|_{L_{q,v}(\Omega)}
\underset{\mathfrak{Z}_1}{\gtrsim} \max \{A_{0,[\tau_-, \,
\tau_+-R]},\, A_{1,[\tau_-, \, \tau_+-R]}\}.
$$
Applying (\ref{go}), the fact that $\varphi$ is Lipschitz, the
inequalities $R\le \overline{R}_{\tau_+}\stackrel
{(\ref{opr2phi})}{\le} \overline{c} \varphi(\tau_+)$ and
$$
\tau_+-R-\tau_-=\tau_+-\underline{R}_{\tau_+}-\tau_-+
(\underline{R}_{\tau_+}-R) \stackrel{(\ref{tau_mt})}{\ge}
\underline{R}_{\tau_+}-R \stackrel{(\ref{tau_mt})}{=}
\left(\frac{1}{\lambda}-1\right)R,
$$
we get that
$$
\max \{A_{0,[\tau_-, \, \tau_+-R]},\, A_{1,[\tau_-, \,
\tau_+-R]}\} \underset{\mathfrak{Z}_1,\lambda}{\gtrsim} \max
\{A_{0,[\tau_-, \, \tau_+]},\, A_{1,[\tau_-, \, \tau_+]}\}.
$$

Prove the upper estimate. By Lemma 2.2 from
\cite{mazya_poborchii}, we may assume that
\begin{align}
\label{cr} \varphi\in C^r(0, \, 1),\quad |\varphi^{(k)}(z)|\le c
\varphi(z)^{1-k}, \quad k\in \overline{1, \, r}, \quad z\in (0, \,
1)
\end{align}
for some $c>0$ not depending on $z$. From (\ref{dgz}) and
(\ref{rz}) it follows that there is $c_*=c_*(\mathfrak{Z}_1)>0$
such that ${\rm dist}(0, \, \partial G_z) \ge  \underline{R}_z\ge
c_*\varphi(z)$. Applying some homothetic transformation of $y$, we
may assume that $\frac{\lambda c_*}{2}=1$. Let $K\in
C^\infty_0(B^{d-1})$, $\int \limits_{B^{d-1}} K(y)\, dy=1$, $\int
\limits _{B^{d-1}}K(y)y^\mu \, dy=0$, $\mu\in \Z_+^{d-1}$, $1\le
|\mu| \le r-1$,
$$
u_\mu(z)=\varphi(z)^{1-d}\int \limits _{|y|<\varphi(z)}
K\left(\frac{y}{\varphi(z)}\right)\frac{\partial
^{|\mu|}u}{\partial y^\mu}(y, \, z)\, dy, \quad z\in (0, \, 1),
\quad u\in W^r_p(\Omega),
$$
$$
Q(x)=\sum \limits _{|\mu|\le r-1} u_\mu(z)\frac{y^\mu}{\mu!},
\quad x=(y, \, z)\in \Omega
$$
(see \cite[page 112]{mazya_poborchii}).

In \cite[Lemma 3.1]{mazya_poborchii} it is proved that if
(\ref{cr}) holds, then the function $u_\mu$ is absolutely
continuous, as well as its derivatives of order not exceeding
$r-1$, $u_\mu^{(r)}\in L_p^{{\rm loc}}(0, \, 1)$, and for
$r-|\mu|\le s\le r$, $z\in (0, \, 1)$
\begin{align}
\label{uas} |u_\mu^{(s)}(z)|\underset{\mathfrak{Z}_1}{\lesssim}
\varphi(z) ^{r-|\mu|-s-\frac{d-1}{p}} \left(\int \limits
_{|y|<\varphi(z)}|\nabla^r u(y, \, z)|^p\, dy\right)^{\frac 1p}.
\end{align}

Let $\mu\in \Z_+^{d-1}$, $|\mu|\le r-1$, $Q_\mu(x)=
u_\mu(z)y^\mu$. Estimate $\|Q_\mu\| _{L_{q,v}(\Omega_{[\tau_-, \,
\tau_+]})}$ from above. Since $u_\mu(z)=0$ in some neighborhood of
$\tau_+$ (see (\ref{wwr_pg})), we have
$$
u_\mu(z)=\frac{(-1)^r}{(r-1)!}\int \limits_z^{\tau_+} (t-z)^{r-1}
u_\mu^{(r)}(t)\, dt, \quad z\in [\tau_-, \, \tau_+].
$$
Applying (\ref{uas}) with $s=r$, we obtain
$$
\|Q_\mu\|_{L_{q,v}(\Omega_{[\tau_-,\, \tau_+]})}
\underset{\mathfrak{Z}_1}{\lesssim} \left(\int \limits
_{\tau_-}^{\tau_+} v_0^q(z)|u_\mu(z)|^q\varphi(z)^{q|\mu|+d-1}\,
dz\right)^{1/q}\le
$$
$$
\le C_{\tau_-, \, \tau_+} \left(\int \limits_{\tau_-}^{\tau_+}
\left| \frac{u_\mu^{(r)}(z)}{g_0(z)}\right|^p \varphi(z)^{p|\mu|
+d-1}\, dz\right)^{\frac 1p}\stackrel{(\ref{uas})}
{\underset{\mathfrak{Z}_1}{\lesssim}} C_{\tau_-, \,
\tau_+}\left(\int \limits_{\tau_-}^{\tau_+} \int
\limits_{|y|<\varphi(z)} \left|\frac{\nabla^r u(y, \, z)}{g(y, \,
z)}\right|^p\, dy\, dz\right)^{\frac 1p}.
$$
The value $C_{\tau_-, \, \tau_+}$ is estimated by applying Theorem
\ref{step_trm}. Since $\varphi$ is non-decreasing, it gives the
desired estimate (see the proof in \cite[page
113]{mazya_poborchii}).

Let us estimate $\|u-Q\|_{L_{q,v}(\Omega_{[\tau_-, \, \tau_+]})}$.
Set $z_0=\tau_+$, $z_{k+1}+ \varphi(z_{k+1}) =z_k$, $k\in \Z_+$,
$k_*=\min\{k\in \Z_+:\, z_{k+1}<\tau_-\}$, $\hat
z_{k+1}=\max\{z_{k+1}, \, \tau_-\}$, $\Omega_{(k)}=\cup_{z\in
[\hat z_{k+1}, \, z_k)}$. By Lemma \ref{un_fca}, there is
$b=b(\mathfrak{Z}_1)>0$ such that $\Omega_{(k)}\in {\bf FC}(b)$,
$k\in \Z_+$. Set $g_{(k)}=g_0(z_k)$, $v_{(k)}=v_0(z_k)$. By
(\ref{go}), $g(x)\underset{\mathfrak{Z}_1}{\asymp} g_{(k)}$,
$v(x)\underset{\mathfrak{Z}_1}{\asymp} v_{(k)}$ for any $x \in
\Omega_{(k)}$. By Theorem \ref{resh_teor} and (\ref{opr2phi}),
$$
\|u\|_{L_q(\Omega_{(k)})}\underset{\mathfrak{Z}_1}{\lesssim}
\varphi(z_k)^{\frac dq-\frac dp}
\left(\|u\|_{L_p(\Omega_{(k)})}+\varphi(z_k)^r \|\nabla^r u\|
_{L_p(\Omega_{(k)})}\right).
$$
Repeating arguments in \cite[page 114]{mazya_poborchii}, we can
prove that
$$
\|u-Q\|_{L_q(\Omega_{(k)})}\underset{\mathfrak{Z}_1}{\lesssim}
\varphi(z_k)^{r+\frac dq-\frac dp}\|\nabla^r
u\|_{L_p(\Omega_{(k)})}.
$$
By construction, $z_k-z_{k+1}=\varphi(z_{k+1})
\underset{\mathfrak{Z}_1}{\asymp} \varphi(z_k)$. Therefore, from
Lemma \ref{kratn_pokr} it follows that
$$
\|u-Q\|_{L_{q,v}(\Omega_{[\tau_-, \, \tau_+]})}
\underset{\mathfrak{Z}_1}{\lesssim} \max_{0\le k\le k_*}
g_{(k)}v_{(k)}\varphi(z_k)^{r+\frac dq-\frac
dp}\left\|\frac{\nabla^r u}{g}\right\|_{L_p(\Omega_{[\tau_-, \,
\tau_+]})} \underset{\mathfrak{Z}_1}{\lesssim}
$$
$$
\lesssim \sup _{t\in [\tau_-, \, \tau_+]} A_{1,[\tau_-, \,
\tau_+]}(t)\left\|\frac{\nabla^r
u}{g}\right\|_{L_p(\Omega_{[\tau_-, \, \tau_+]})}
$$
(the last relation is similar to the inequality (4.7) in
\cite{mazya_poborchii}; it follows from the monotonicity of
$\varphi$).
\end{proof}
\renewcommand{\proofname}{\bf Proof}

Denote by ${\cal P}_{r-1}(\R^d)$ the space of polynomials on
$\R^d$ of degree not exceeding $r-1$. For each measurable subset
$E\subset \R^d$ we put
$${\cal P}_{r-1}(E)= \{f|_E:\, f\in {\cal P}_{r-1}(\R^d)\}.$$

\begin{Cor}
\label{pol_appr} Let $\Omega \in {\bf FC}_\varphi(a)$,
$0<\lambda_*<1$, $0\le \tau_-<\tau_+\le \frac 12$, and let
\begin{align}
\label{ttttt} \tau_-<\tau_+-\lambda_*\varphi(\tau_+).
\end{align}
Suppose that $A_{[\tau_-, \, \tau_+]}<\infty$. Then
$W^r_{p,g}(\Omega_{[\tau_-, \, \tau_+]}) \subset
L_{q,v}(\Omega_{[\tau_-, \, \tau_+]})$ and there exists a linear
continuous operator
$$P:L_{q,v}(\Omega_{[\tau_-, \, \tau_+]}) \rightarrow {\cal
P}_{r-1}(\Omega_{[\tau_-, \, \tau_+]})$$ such that for any
function $f\in W^r_{p,g}(\Omega_{[\tau_-, \, \tau_+]})$
\begin{align}
\label{ap_pol} \|f-Pf\|_{L_{q,v}(\Omega_{[\tau_-, \, \tau_+]})}
\underset{\mathfrak{Z}_1}{\lesssim} A_{[\tau_-, \, \tau_+]}
\left\|\frac{\nabla^r f}{g}\right\| _{L_p(\Omega_{[\tau_-, \,
\tau_+]})}.
\end{align}
\end{Cor}
\begin{proof}
Let $\tau_-\ge\tau_+-\underline{R}_{\tau_+}$. Then $\tau_+ -\tau_-
\stackrel{(\ref{rrr}),(\ref{opr2phi})}{\le} c_0\varphi(\tau_+)
\underset{\mathfrak{Z}_1}{\lesssim} \varphi(\tau_-)$ (here
$c_0=c_0(\mathfrak{Z}_1)$). Prove the last inequality. Since
$\varphi'(t) \underset{t\to +0}{\to 0}$, there exists
$\hat\tau=\hat\tau(\mathfrak{Z}_1)>0$ such that for
$\tau_+<\hat\tau$ we get $\varphi(\tau_-)\ge \varphi(\tau_+-
c_0\varphi(\tau_+)) \ge \frac{\varphi(\tau_+)}{2}$. If $\tau_+
>\hat \tau$, then $\underline{R}_{\tau_+}\le \overline{R}
_{\tau_+} \stackrel{(\ref{opr2k})}{\le} \kappa_\Omega \tau_+$ and
$\tau_- \ge \frac{\tau_+}{1-\kappa_\Omega} \ge \frac{\hat \tau}
{1-\kappa_\Omega}$. Therefore, $\varphi(\tau_-) \underset
{\mathfrak{Z}_1}{\asymp} 1 \underset{\mathfrak{Z}_1}{\asymp}
\varphi(\tau_+)$.

Thus, $\tau_+ -\tau_-\underset{\mathfrak{Z}_1}{\lesssim}
\varphi(\tau_-)$. By Lemma \ref{un_fca}, $\Omega_{[\tau_-, \,
\tau_+]} \in {\bf FC}(\tilde b)$, $\tilde b
\underset{\mathfrak{Z}_1}{\asymp} 1$. It remains to apply Lemma
\ref{3_prop}, Theorem \ref{resh_teor} and the definition of
$A_{[\tau_-, \, \tau_+]}$.

Suppose that $\tau_-<\tau_+-\underline{R}_{\tau_+}$. Let $P_0:{\rm
span}\, W^r_p(B_{\underline{R}_{\tau_+}} (\eta_{\tau_+}))
\rightarrow {\cal P}_{r-1}
(B_{\underline{R}_{\tau_+}}(\eta_{\tau_+}))$ be a linear
continuous projection. Then for any $0\le k\le r$ and for each
$q_k\in [1, \, +\infty)$ such that $r-k+\frac{d}{q_k}-\frac dp>0$
we have
\begin{align}
\label{qk} \displaystyle
\begin{array}{c}
\|f-P_0f\|_{L_{q_k}(B_{\underline{R}_{\tau_+}}(\eta_{\tau_+}))}
\underset {p,q_k,r,d} {\lesssim} \underline{R}_{\tau_+}
^{r-k+\frac{d}{q_k}-\frac{d}{p}} \|\nabla^r f\|
_{L_p(B_{\underline{R}_{\tau_+}}(\eta_{\tau_+}))}
\stackrel{(\ref{opr2phi})}{\underset{\mathfrak{Z}_1}{\lesssim}} \\
\lesssim (\varphi(\tau_+))^{r-k+\frac{d}{q_k}-\frac{d}{p}}
\|\nabla^r f\| _{L_p(B_{\underline{R}_{\tau_+}}(\eta_{\tau_+}))}
\end{array}
\end{align}
(see \cite{sobolev1}).

In order to define $P_0$, we take the orthogonal projection in
$L_2(B_{\underline{R}_{\tau_+}}(\eta_{\tau_+}))$ onto the space
${\cal P}_{r-1}(B_{\underline{R}_{\tau_+}}(\eta_{\tau_+}))$,
extend it on $L_1(B_{\underline{R}_{\tau_+}}(\eta_{\tau_+}))$ by
continuity, and then restrict it to $L_{q,v}
(B_{\underline{R}_{\tau_+}} (\eta_{\tau_+}))$. The function $Pf$
is defined as the extension of the polynomial $P_0(f)$ on the
domain $\Omega_{[\tau_-, \, \tau_+]}$. Then the image of $P$ is
contained in ${\cal P}_{r-1}(\Omega_{[\tau_-, \, \tau_+]})$. From
the condition $A_{[\tau_-, \, \tau_+]}<\infty$ it follows that
$v\in L_q(\Omega)$ and the operator $P:L_{q,v}(\Omega_{[\tau_-, \,
\tau_+]})\rightarrow L_{q,v}(\Omega_{[\tau_-, \, \tau_+]})$ is
continuous.

Let $\psi_0\in C_0^\infty(\R^d)$, ${\rm supp}\, \psi_0\subset
B_1(0)$, $\psi_0|_{B_{1/2}(0)}=1$, $\psi_0(x)\in [0, \, 1]$ for
any $x\in \R^d$. Set $R=\frac{\underline{R}_{\tau_+}}{2}$,
$\psi(x)=\psi_0\left(\frac{x-\eta_{\tau_+}}{\underline{R}_{\tau_+}}\right)$.
Then ${\rm supp}\, \psi\subset B_{\underline{R}_{\tau_+}}
(\eta_{\tau_+})$, ${\rm supp}\, (1-\psi)\subset \Omega \backslash
B_R (\eta_{\tau_+})$. By (\ref{opr2k}), (\ref{opr2phi}),
(\ref{go}) and Lemma \ref{3_prop},
\begin{align}
\label{gvv} \frac{g(x)}{g(y)} \underset{\mathfrak{Z}_1}{\asymp} 1,
\quad \frac{v(x)}{v(y)} \underset{\mathfrak{Z}_1}{\asymp} 1, \quad
x, \, y\in B_{\underline{R}_{\tau_+}} (\eta_{\tau_+}).
\end{align}
Hence,
$$
\|\psi(f-Pf)\|_{L_{q,v} (\Omega_{[\tau_-, \, \tau_+]})}\le
\|f-Pf\|_{L_{q,v}(B_{\underline{R}_{\tau_+}} (\eta_{\tau_+}))}
\stackrel{(\ref{qk}),(\ref{gvv})}{\underset{\mathfrak{Z}_1}{\lesssim}}
$$
$$
\lesssim \sup _{x\in B_{\underline{R}_{\tau_+}} (\eta_{\tau_+})}
g(x)v(x) (\varphi(\tau_+))^\delta \left\|\frac{\nabla^r
f}{g}\right\| _{L_p(B_{\underline{R}_{\tau_+}}(\eta_{\tau_+}))}
\underset{\mathfrak{Z}_1}{\lesssim} A_{[\tau_-, \, \tau_+]}
\left\|\frac{\nabla^r f}{g}\right\| _{L_p(\Omega_{[\tau_-, \,
\tau_+]})}
$$
(the last inequality follows from the monotonicity of the function
$\varphi$; see the end of proof of Theorem \ref{emb_trm_peak}).
The value $\|(1-\psi)(f-Pf)\|_{L_{q,v}(\Omega_{[\tau_-, \,
\tau_+]})}$ is estimated by Theorem \ref{emb_trm_peak} (the
arguments are the same as in \cite{vas_john}).
\end{proof}

Let (\ref{g0v0}), (\ref{phi_def}), (\ref{limtrho}),
(\ref{sigma_prop}), (\ref{alpha_prop}), (\ref{rho_def}) hold. We
claim that
\begin{align}
\label{at} A_{[\tau_-, \, \tau_+]}\le A_{[0, \, \tau_+]} \underset
{\mathfrak{Z}_1} {\lesssim} |\log \tau_+|^{-\alpha}\rho(|\log
\tau_+|).
\end{align}
Indeed, the integrals in the definition of $A_{i,[0, \, \tau_+]}$
($i=0, \, 1$) can be easily estimated (to this end we employ Lemma
\ref{sum_lem} and replace $z-t$ by $z$ and $t-z$ by $t$). Then we
apply the fact that the function $t^{-\alpha}\rho(t)$ is
decreasing for large $t$.

This estimate will be employed for $\tau_-< \frac{\tau_+}{2}$. If
$\tau_-\ge \frac{\tau_+}{2}$, then
$\varphi(t)\underset{\mathfrak{Z}_1}{\asymp} \varphi(\tau_+)$,
$g_0(t)\underset{\mathfrak{Z}_1}{\asymp} g_0(\tau_+)$,
$v_0(t)\underset{\mathfrak{Z}_1}{\asymp} v_0(\tau_+)$ for any
$t\in [\tau_-, \, \tau_+]$. Therefore,
\begin{align}
\label{at1} \begin{array}{c} A_{[\tau_-, \, \tau_+]}\underset
{\mathfrak{Z}_1} {\lesssim} g_0(\tau_+)v_0(\tau_+)
[\varphi(\tau_+)]^{(d-1) \left(\frac 1q-\frac 1p\right)}
(\tau_+-\tau_-)^{r+\frac 1q-\frac 1p}=\\ =|\log
\tau_+|^{-\alpha}\rho(|\log \tau_+|) \tau_+^{-r-\frac 1q +\frac
1p}(\tau_+-\tau_-)^{r+\frac 1q-\frac 1p}. \end{array}
\end{align}

\section{Estimates of widths}
In this section, we suppose that $p\le q$.

Let $G\subset \Omega$ be a measurable set, and let $T=\{G
_i\}_{i=1}^{i_0}$ be a finite partition of $G$. Denote
\begin{align}
\label{calsrtgsgr} {\cal S}_{r,T}(G)=\{S:\Omega\rightarrow \R:
S|_{G _i}\in {\cal P}_{r-1}(G _i), \; 1\le i \le i_0, \;\;
S|_{\Omega \backslash G}=0\};
\end{align}
for each $f\in L_{q,v}(G)$ we set
\begin{align}
\label{norm_f_pqtv} \| f\| _{p,q,T,v}= \left(\sum \limits
_{i=1}^{i_0} \| f\| _{L_{q,v}(G_i)}^p\right) ^{\frac{1}{p}}.
\end{align}
By $L_{p,q,T,v}(G)$ we denote the space $L_{q,v}(G)$ equipped
with the norm $\| \cdot\| _{p,q,T,v}$. Notice that $\| f\| _{p,q,T,v}\ge
\| f\| _{L_{q,v}(G)}$.

The following assertion (in fact, a more general result for
weighted spaces) was proved in \cite{vas_john}.  For the
non-weighted case, Besov \cite{besov_peak_width} later put forward a
more simple proof.
\begin{Lem}
\label{m_part} Let $a>0$, $G\subset \R^d$, $G\in {\bf FC}(a)$,
$n\in \N$. Then there exists a family of partitions
$\{T_{m,n}(G)\}_{m\in \Z_+}$ with the following properties:
\begin{enumerate}
\item ${\rm card}\, T_{m,n}(G)\underset{a,d}{\lesssim}2^mn$;
\item for any $E\in T_{m,n}(G)$ there exists a linear continuous
operator $P_E:L_q(E)\rightarrow {\cal P}_{r-1}(E)$ such that for
any function $f\in W^r_p(G)$
\begin{align}
\label{lq_e} \|f-P_Ef\|_{L_q(E)}\underset{p,q,r,d,a}{\lesssim}
(2^mn)^{-\frac rd-\frac 1q+\frac 1p} ({\rm mes}\, G)^{\frac
rd+\frac 1q-\frac 1p} \|\nabla^r f\|_{L_p(E)};
\end{align}
\item for any $m\in \Z_+$, $E\in T_{m,n}(G)$
$$
{\rm card}\, \{E'\in T_{m+1,n}(G):\; {\rm mes}\, (E\cap
E')>0\}\underset{a,d}{\lesssim}1,
$$
$$
{\rm card}\, \{E'\in T_{m-1,n}(G):\; {\rm mes}\, (E\cap
E')>0\}\underset{a,d}{\lesssim}1, \text{ if }m\ge 1.
$$
\end{enumerate}
\end{Lem}

Let $X$, $Y$ be normed spaces, $B\subset X$, $A\in L(X, \, Y)$. Then
\begin{align}
\label{width_op} d_n(A(B), \, Y)\le \|A\|d_n(B, \, X).
\end{align}
If $A$ is an isomorphism of $X$ and $Y$, then
\begin{align}
\label{width_op_lin} \lambda_n(A(B), \, Y)\le \|A\|\lambda _n(B,
\, X), \quad  d^n(A(B), \, Y)\le \|A\|d^n(B, \, X).
\end{align}
The same inequalities hold if $Y$ is a subspace in
$X$, $B\subset Y$ and $A$ is a linear projection onto $Y$.
These assertions follow from definitions of Kolmogorov,
Gelfand and linear widths.

Denote by $\chi_E$ the indicator function of a set
$E$.

\begin{Lem}
\label{low} Let $\Omega\subset \R^d$ be a domain, let $G_1, \,
\dots, \, G_m\subset \Omega$ be pairwise non-overlapping
domains, and let $\psi_1, \, \dots, \, \psi_m\in W^r_{p,g}(\Omega)$,
$\left\|\frac{\nabla^r\psi_j}{g}\right\|_{L_p(\Omega)}=1$, ${\rm
supp} \, \psi_j\subset G_j$, $\|\psi_j\|_{L_{q,v}(G_j)}\ge M$,
$1\le j\le m$. Then
$$
\vartheta_n(W^r_{p,g}(\Omega), \, L_{q,v}(\Omega)) \ge M\cdot
\vartheta_n(B^m_p, \, l^m_q)
$$
(see notations on page \pageref{vartheta_n_def}).
\end{Lem}
\begin{proof}
Let $X={\rm span}\, \{\psi_j\}_{j=1}^m\subset L_{q,v}(\Omega)$.
From the definition of Gelfand widths it follows that
$$
d^n(W^r_{p,g}(\Omega), \, L_{q,v}(\Omega)) \ge
d^n(W^r_{p,g}(\Omega)\cap X, \, X).
$$
Prove that
$$
d_n(W^r_{p,g}(\Omega), \, L_{q,v}(\Omega)) \ge
d_n(W^r_{p,g}(\Omega)\cap X, \, X),$$$$
\lambda_n(W^r_{p,g}(\Omega), \, L_{q,v}(\Omega)) \ge
\lambda_n(W^r_{p,g}(\Omega)\cap X, \, X).
$$
To this end we construct the projection $Q:L_{q,v}(\Omega)
\rightarrow X$ such that $\|Q\|\le 1$ and apply
(\ref{width_op}), (\ref{width_op_lin}). Let
$X_j={\rm span}\, \{\psi_j\}$. Denote by $L_{q,v}^{(j)}(G)$
the set of functions in $L_{q,v}(G)$ whose support is
contained in $G_j$. Since $\dim X_j=1$, there is a projection
$Q_j:L_{q,v}^{(j)}(G)\rightarrow X_j$, $\|Q_j\|\le 1$. Let
$Q(f)=\sum \limits_{j=1}^m Q_j(f\chi_{G_j})$. Since the set
$G_j$ do not overlap pairwise, we get $\|Q\|\le 1$.

Define the isomorphism $T:X\rightarrow \R^m$ by
$$T\left(\sum \limits_{j=1}^m c_j\psi_j\right)=(c_1, \, \dots, \,
c_m).$$ Since $\left\|\frac{\nabla^r \psi_j}{g}\right\|
_{L_p(\Omega)} =1$, $\sum \limits_{j=1}^m c_j\psi_j\in
W^r_{p,g}(\Omega)$ holds if and only if $(c_1, \, \dots, \,
c_n)\in B_p^m$. Therefore, $T(W^r_{p,g}(\Omega)\cap X)=B_p^m$.
Prove that $\|T\|_{L_{q,v}(G)\rightarrow l_q^m}\le M^{-1}$. Indeed,
if $f=\sum \limits_{j=1}^m c_j\psi_j$, then
$\|Tf\|_{l_q^m}=\left(\sum \limits_{j=1}^m |c_j|^q\right)^{1/q}$,
$$
\|f\|_{L_{q,v}(G)}=\left(\sum \limits_{j=1}^m |c_j|^q
\|\psi_j\|^q_{L_{q,v}(G_j)}\right)^{1/q}\ge M\|Tf\|_{l_q^m}.
$$
By (\ref{width_op}) and (\ref{width_op_lin}),
$$
\vartheta_n(B_p^m, \, l_q^m)=\vartheta_n(T(W^r_{p,g}(\Omega)\cap
X), \, l_q^m)\le $$$$\le\|T\|_{L_{q,v}(G)\rightarrow
l_q^m}\vartheta_n(W^r_{p,g}(\Omega)\cap X, \, X)\le M^{-1}
\vartheta_n(W^r_{p,g}(\Omega)\cap X, \, X).
$$
This implies the desired estimates.
\end{proof}


\begin{Sta}
\label{pokr_razb} Let $M>0$, $m_0\in \Z_+\cup\{+\infty\}$ and let
$\{T_m\}_{m=0}^{m_0}$ be a family of finite coverings of a domain
$G$ with the following properties:
\begin{enumerate}
\item ${\cal N}_{T_m}\le M$ for any $m\in \overline{0, \, m_0}$;
\item for any $m\in \overline{0, \, m_0}$, $E\in T_m$ there exists a linear
continuous operator $P_{E,m}:L_{q,v}(E)\rightarrow {\cal
P}_{r-1}(E)$ such that for any function $f\in W^r_{p,g}(G)$ the
inequality $\|f-P_{E,m}f\|_{L_{q,v}(E)}\le
C_m\left\|\frac{\nabla^r f}{g}\right\|_{L_p(E)}$ holds;
\item ${\rm card}\, \{E'\in T_{m\pm 1}:\, {\rm mes}
(E\cap E')>0\}\le M$ for any $E\in T_m$.
\end{enumerate}
Let $U\subset G$ be a measurable subset. Then there exists a
sequence of partitions $\{\hat T_m\}_{m=0}^{m_0}$ of the set $U$
such that
\begin{enumerate}
\item ${\rm card}\, \hat T_m\le {\rm card}\, T_m$;
\item there exists a linear continuous operator
$P_m:L_{q,v}(G) \rightarrow {\cal S}_{r,\hat T_m}(U)$ such that
for any function $f\in W^r_{p,g}(G)$ the inequality
$\|f-P_mf\|_{p,q,\hat T_m,v}\underset{M,p}{\lesssim} C_m
\left\|\frac{\nabla^r f}{g}\right\|_{L_p(G)}$ holds;
\item ${\rm card}\, \{E'\in \hat T_{m\pm 1}:\, {\rm mes}
(E\cap E')>0\}\le M$ for any $E\in \hat T_m$;
\item there are injective mappings ${\cal F}_m: \hat T_m \rightarrow T_m$ such that
for any $E\in \hat T_m$ the inclusion $E \subset {\cal F}_m(E)$
holds.
\end{enumerate}
\end{Sta}
\begin{proof}
Let $T_m=\{E_{i,m}\}_{i=1}^{k_m}$. Denote
$$
\hat E_{1,m}=U\cap E_{1,m}, \quad \hat E_{i,m}=U\cap
E_{i,m}\backslash \cup_{j=1}^{i-1} E_{j,m}, \quad i\ge 2,
$$
$$
\hat T_m=\{\hat E_{i,m}\}_{i\in I_m}, \quad I_m=\{i\in
\overline{1, \, k_m}:\; {\rm mes}\, \hat E_{i,m}>0\},
$$
${\cal F}_m(\hat E_{i,m})=E_{i,m}$, $i\in I_m$. Then $\hat T_m$
satisfies assertions 1, 3 and 4. Define the operator $P_m$ by
$P_mf|_{\hat E_{i,m}}=P_{E_{i,m},m}f|_{\hat E_{i,m}}$, $i\in I_m$.
Apply the condition ${\cal N}_{T_m}\le M$ and obtain assertion 2.
\end{proof}

\renewcommand{\proofname}{\bf Proof of Theorem \ref{width_est}}
\begin{proof}
{\it Proof of the upper estimate.} It suffices to consider the
case $n=2^{Nd}$, $N\in \N$. Denote
$$
D_1=\Omega _{[2^{-2}, \, 2^{-1}]}, \quad D_j= \Omega _{[2^{-j-1},
\, 2^{-j}]} \backslash \cup _{i=1}^{j-1} \Omega _{[2^{-i-1}, \,
2^{-i}]}, \quad j\ge 2.
$$

{\bf Step 1.} Let $j\in \N$, $j\ge 2$, $l\in \Z_+$. Construct the
covering $\tilde {\bf R}_{j,l}$ of $\Omega_{[2^{-j},2^{-j+1}]}$.
Without loss of generality we may assume that
\begin{align}
\label{phi_z_z} \varphi(z)\le z \quad \text{ for any } \quad z\in
(0, \, 1).
\end{align}
Choose $l_j\in \Z_+$ so that
\begin{align}
\label{2jlj1_phi_2j} 2^{-j-l_j-1}<\varphi(2^{-j})\le 2^{-j-l_j}.
\end{align}
For $0\le l\le l_j$, $1\le i\le 2^l$ we set
$\tau_{j,l}(i)=2^{-j}+i\cdot 2^{-j-l}$, $\tilde {\bf R}_{j,l}=
\{U_{i,j,l}\}_{1\le i\le 2^l}$, $U_{i,j,l} = \Omega
_{[\tau_{j,l}(i-1), \, \tau_{j,l}(i)]}$. Then ${\rm card}\, \tilde
{\bf R}_{j,l}=2^l$; by Lemma \ref{kratn_pokr},
\begin{align}
\label{kr_pokr} {\cal N}_{\tilde {\bf R}_{j,l}}
\underset{\mathfrak{Z}}{\lesssim} 1;
\end{align}
this together with the definition of $\tilde{\bf R}_{j,l}$ implies
that
$$
{\rm card}\, \{U'\in \tilde {\bf R}_{j,l\pm 1}:\; {\rm mes}\,
(U'\cap U)>0\} \underset{\mathfrak{Z}}{\lesssim} 1.
$$
By Corollary \ref{pol_appr} and (\ref{at1}), for any $i\in
\overline{1, \, 2^l}$ there exists a linear continuous operator
$P_i^{j,l}:L_{q,v}(U_{i,j,l}) \rightarrow {\cal
P}_{r-1}(U_{i,j,l})$ such that for any function $f\in
W^r_{p,g}(\Omega)$
\begin{align}
\label{a1777} \displaystyle \begin{array}{c}
\|f-P_i^{j,l}f\|_{L_{q,v}(U_{i,j,l})}
\underset{\mathfrak{Z}}{\lesssim} j^{-\alpha} \rho(j)\cdot
2^{-l\left(r+\frac 1q-\frac 1p\right)}\left\| \frac{\nabla^r
f}{g}\right\|_{L_p(U_{i,j,l})}\le \\ \le j^{-\alpha} \rho(j)\cdot
2^{-l\left(\frac rd+\frac 1q-\frac 1p\right)}\left\|
\frac{\nabla^r f}{g}\right\|_{L_p(U_{i,j,l})}.
\end{array}
\end{align}

By Proposition \ref{pokr_razb}, for any $l=\overline{0, \, l_j}$
there exist a partition ${\bf R} _{j,l}$ of $D_{j-1}$, an
injective mapping ${\cal F}_{j,l}:{\bf R} _{j,l} \rightarrow
\tilde {\bf R} _{j,l}$ and a linear continuous operator
$\mathbb{P}_{j,l}:L_{q,v}(\Omega _{[2^{-j}, \, 2^{-j+1}]})
\rightarrow {\cal S}_{r,{\bf R} _{j,l}}(D_{j-1})$ such that
\begin{align}
\label{ceij} {\rm card}\, {\bf R}_{j,l}=2^l,
\end{align}
\begin{align}
\label{fx468jhdewml_8h} {\rm card}\, \{U'\in {\bf R}_{j,l\pm 1}:\;
{\rm mes}\, (U'\cap U)>0\} \underset{\mathfrak{Z}}{\lesssim} 1
\end{align}
and for any $f\in W^r_{p,g}(\Omega _{[2^{-j}, \, 2^{-j+1}]})$
\begin{align}
\label{a1} \displaystyle \begin{array}{c}
\|f-\mathbb{P}_{j,l}f\|_{p,q,{\bf R}_{j,l},v}
\underset{\mathfrak{Z}}{\lesssim} j^{-\alpha} \rho(j)\cdot
2^{-l\left(\frac rd+\frac 1q-\frac 1p\right)}\left\|
\frac{\nabla^r f}{g}\right\|_{L_p(\Omega _{[2^{-j}, \,
2^{-j+1}]})}.
\end{array}
\end{align}

By Lemma \ref{un_fca} and (\ref{2jlj1_phi_2j}), $U_{i,j,l_j}\in
{\bf FC}(b_*)$ with $b_*=b_*(\mathfrak{Z})>0$ and
\begin{align}
\label{mes_u_ijl} {\rm mes}\, U_{i,j,l_j} \underset{\mathfrak{Z}}
{\asymp} \varphi^d(2^{-j}).
\end{align}
Hence, by Lemma \ref{m_part}, for any $l\ge l_j$ there exists a
partition $T_{l,j,i}=T_{l-l_j,1}(U_{i,j,l_j})$ such that
\begin{enumerate}
\item ${\rm card}\, T_{l,j,i}\underset{b_*,d}{\lesssim}
2^{l-l_j}$;
\item for any $E\in T_{l,j,i}$ there exists a linear continuous
operator $P_E:L_q(E) \rightarrow {\cal P}_{r-1}(E)$ such that for
any function $f\in W^r_p(\Omega)$
\begin{align}
\label{8888} \|f-P_Ef\|_{L_q(E)} \underset{p,q,r,d,b_*}{\lesssim}
2^{-(l-l_j)\left(\frac rd+\frac 1q-\frac 1p\right)} ({\rm mes}\,
U_{i,j,l_j})^{\frac rd+\frac 1q-\frac 1p} \|\nabla^r f\|_{L_p(E)};
\end{align}
\item for any $l\ge l_j$, $G\in T_{l,j,i}$
$$
{\rm card}\, \{G'\in T_{l+1,j,i}:\; {\rm mes}\, (G\cap
G')>0\}\underset{b_*,d}{\lesssim}1,
$$
$$
{\rm card}\, \{G'\in T_{l-1,j,i}:\; {\rm mes}\, (G\cap
G')>0\}\underset{b_*,d}{\lesssim}1, \text{ if }l\ge l_j+1.
$$
\end{enumerate}
From (\ref{g0v0}),  (\ref{phi_def}), (\ref{limtrho}),
(\ref{sigma_prop}), (\ref{alpha_prop}), (\ref{rho_def}),
(\ref{2jlj1_phi_2j}), (\ref{mes_u_ijl}), (\ref{8888}) and the
inequality $2^{\frac{\delta}{d}l_j} \le 2^{\left(r+\frac 1q-\frac
1p\right)l_j}$ it follows that for any $E\in T_{l,j,i}$
\begin{align}
\label{9999} \|f-P_Ef\|_{L_{q,v}(E)}
\underset{\mathfrak{Z}}{\lesssim} j^{-\alpha}\rho(j)\cdot
2^{-l\left(\frac rd+\frac 1q-\frac 1p\right)}\left\|\frac{\nabla^r
f}{g}\right\|_{L_p(E)}.
\end{align}

Let $D\in {\bf R}_{j,l_j}$. Denote by $i(D)$ the number $i$ such
that $U_{i,j,l_j} ={\cal F}_{j,l}(D)$. For $l>l_j$ we set
\begin{align}
\label{r_jl} {\bf R}_{j,l}=\{D\cap E: \quad D\in {\bf R}_{j,l_j},
\;\; E\in T_{l,j,i(D)}\}.
\end{align}
Then ${\bf R}_{j,l}$ is a partition of $D_{j-1}$. From
(\ref{ceij}) and property 1 of the partition $T_{l,j,i}$ it
follows that
\begin{align}
\label{zzz} {\rm card}\, {\bf R}_{j,l}\underset{b_*,d}{\lesssim}
2^l.
\end{align}
Prove that
\begin{align}
\label{ooo} {\rm card}\, \{U'\in {\bf R}_{j,l\pm 1}:\; {\rm mes}\,
(U'\cap U)>0\} \underset{\mathfrak{Z}}{\lesssim} 1.
\end{align}
Indeed, let $D, \, D'\in {\bf R}_{j,l_j}$, $E\in T_{l,j,i(D)}$,
$E' \in T_{l\pm 1,j,i(D')}$, ${\rm mes}\, (D\cap E\cap D'\cap
E')>0$. Since ${\bf R}_{j,l_j}$ is a partition, we get $D'=D$. It
remains to apply property 3 of the partition $T_{l,j,i}$.

Let $f\in L_{q,v}(\Omega _{[2^{-j}, \, 2^{-j+1}]})$. Set
$\mathbb{P}_{j,l}f|_{D\cap E} =P_Ef|_{D\cap E}$ with $D$, $E$ from
(\ref{r_jl}) and $P_E$ from property 2. Denote $C_{j,l} =
j^{-\alpha}\rho(j)\cdot 2^{-l\left(\frac rd+\frac 1q-\frac
1p\right)}$. Then $\mathbb{P}_{j,l}:L_{q,v}(\Omega _{[2^{-j}, \,
2^{-j+1}]}) \rightarrow {\cal S}_{r,{\bf R}_{j,l}}(D_{j-1})$ is a
linear continuous operator and for any function $f\in
W^r_{p,g}(\Omega _{[2^{-j}, \, 2^{-j+1}]})$
$$
\|f-\mathbb{P}_{j,l}f\|_{p,q,{\bf R}_{j,l},v} =\left( \sum \limits
_{D\in {\bf R}_{j,l_j}} \sum \limits _{E\in T_{l,j,i(D)}}
\|f-P_Ef\|^p _{L_{q,v}(D\cap E)}\right)^{\frac 1p}
\stackrel{(\ref{9999})}{\underset{\mathfrak{Z}}{\lesssim}}
$$
$$
\lesssim C_{j,l}\left( \sum \limits _{D\in {\bf R}_{j,l_j}} \sum
\limits _{E\in T_{l,j,i(D)}} \left\| \frac{ \nabla ^r
f}{g}\right\|^p_{L_p(E)}\right)^{\frac 1p}=C_{j,l}\left( \sum
\limits _{D\in {\bf R}_{j,l_j}} \left\| \frac{ \nabla ^r
f}{g}\right\|^p_{L_p(U_{i(D),j,l_j})}\right)^{\frac 1p}\le
$$
$$
\le C_{j,l}\left( \sum \limits _{i=1}^{l_j} \left\| \frac{ \nabla
^r f}{g}\right\|^p_{L_p(U_{i,j,l_j})}\right)^{\frac 1p}
\stackrel{(\ref{kr_pokr})}{\underset{\mathfrak{Z}}{\lesssim}}
C_{j,l}\left\| \frac{ \nabla ^r f}{g}\right\|^p_{L_p(\Omega
_{[2^{-j}, \, 2^{-j+1}]})}
$$
(here we used that ${\cal F}_{j,l_j}$ is an injection). Thus,
\begin{align} \label{a2}
\|f-\mathbb{P}_{j,l}f\|_{p,q,{\bf R}_{j,l},v}
\underset{\mathfrak{Z}}{\lesssim} j^{-\alpha}\rho(j)\cdot
2^{-l\left(\frac rd+\frac 1q-\frac 1p\right)}\left\|\frac{\nabla^r
f}{g}\right\|_{L_p(\Omega _{[2^{-j}, \, 2^{-j+1}]})}.
\end{align}

{\bf Step 2.} For $0\le t\le Nd$ we set
$G_t:=\Omega_{[2^{-2^{t+1}},2^{-2^t}]}$, $$\hat G_0=G_0, \quad
\hat G_t =G_t \backslash \cup_{s=0}^{t-1} G_s.$$ Then
$\cup_{t=0}^\infty G_t=\cup_{t=0}^\infty \hat G_t=\Omega$. The
family of domains $\{\Omega_{[2^{-j}, \, 2^{-j+1}]}\}_{2^t+1\le
j\le 2^{t+1}}$ forms a covering $Q_{(t)}$ of the set $G_t$. By
Lemma \ref{kratn_pokr},
\begin{align}
\label{nqt_z_1} {\cal N} _{Q_{(t)}} \underset {\mathfrak{Z}}
{\lesssim} 1.
\end{align}
The family $\{D_{j-1}\}_{2^t+1\le j\le 2^{t+1}}$ forms a partition
of $\hat G_t$.

Let $t_*=t_*(N)=0$ or $t_*=t_*(N)=Nd$ for each $N$, and let
$\varepsilon>0$. The choice of $\varepsilon$ and $t_*$ (both
dependent on $\mathfrak{Z}$) will be made later. Denote
$$
m_t^*=\max\{\lceil t-Nd+\varepsilon|t-t_*|\rceil, \, 0\},
$$
\begin{align}
\label{lmt} l_{m,t}=\lceil Nd-t-\varepsilon |t-t_*|\rceil+m, \quad
m\in \Z_+, \quad m\ge m_t^*,
\end{align}
\begin{align}
\label{tmtn1_def_1290} \hat T_{m,t,n}^1=\{U\in {\bf
R}_{j,l_{m,t}}, \; 2^t+1\le j\le 2^{t+1}\}.
\end{align}
Then $\hat T_{m,t,n}^1$ is a partition of the set $\hat G_t$. By
(\ref{fx468jhdewml_8h}) and (\ref{ooo}), for any $U\in \hat
T_{m,t,n}^1$
\begin{align}
\label{custmtn} {\rm card}\, \{U'\in \hat T_{m\pm 1,t,n}^1:\, {\rm
mes}\, (U\cap U')\ne 0\} \underset{\mathfrak{Z}}{\lesssim} 1.
\end{align}
From (\ref{ceij}), (\ref{zzz}) and (\ref{tmtn1_def_1290}) it
follows that
\begin{align}
\label{ctmnt} {\rm card}\, \hat T^1_{m,t,n}
\underset{\mathfrak{Z}} {\lesssim} 2^t\cdot 2^{l_{m,t}}\asymp
n\cdot 2^{m-\varepsilon|t-t_*|}.
\end{align}

For $f\in L_{q,v}(\Omega)$ we set
$$
P^1_{m,t,n}f|_{D_{j-1}} =\mathbb{P}_{j,l_{m,t}}f|_{D_{j-1}}, \quad
2^t+1\le j\le 2^{t+1}, \quad P^1_{m,t,n}f|_{\Omega \backslash \hat
G_t}=0.
$$
Then $P^1_{m,t,n}: L_{q,v}(\Omega) \rightarrow {\cal S}_{r, \hat
T^1_{m,t,n}}(\hat G_t)$ is a linear continuous operator. For any
function $f\in W^r_{p,g}(\Omega)$ we have
$$
\|f-P_{m,t,n}^1f \|_{p,q,\hat T_{m,t,n}^1,v} =\left( \sum \limits
_{j=2^t+1}^{2^{t+1}}\|f-\mathbb{P}_{j,l_{m,t}}f \|^p_{p,q,{\bf
R}_{j, l_{m,t}}, v}\right)^{\frac 1p} \stackrel{(\ref{a1}),
(\ref{a2})}{ \underset{\mathfrak{Z}}{\lesssim}}
$$
$$
\lesssim \left( \sum \limits _{j=2^t+1}^{2^{t+1}}
j^{-p\alpha}\rho^p(j)\cdot 2^{-\frac{p\delta}{d}l_{m,t}} \left\|
\frac{\nabla ^rf}{g}\right\|^p_{L_p(\Omega _{[2^{-j}, \,
2^{-j+1}]})} \right)^{\frac 1p}\underset{\mathfrak{Z}}{\lesssim}
$$
$$
\lesssim 2^{-t\alpha}\rho(2^t) 2^{-\frac{\delta}{d}l_{m,t}}\left(
\sum \limits _{j=2^t+1}^{2^{t+1}} \left\| \frac{\nabla
^rf}{g}\right\|^p_{L_p(\Omega _{[2^{-j}, \, 2^{-j+1}]})}
\right)^{\frac 1p} \stackrel{(\ref{nqt_z_1})} {\underset
{\mathfrak{Z}} {\lesssim}} 2^{-t\alpha}\rho(2^t)
2^{-\frac{\delta}{d}l_{m,t}}.
$$
Thus, for any function $f\in W^r_{p,g}(\Omega)$
\begin{align}
\label{appr} \|f-P_{m,t,n}^1f\|_{p,q,\hat T_{m,t,n}^1,v}
\underset{\mathfrak{Z}}{\lesssim} 2^{-t\alpha}\rho(2^t)
2^{-\frac{\delta}{d}l_{m,t}}.
\end{align}

{\bf Step 3.} Given $t\in \Z_+$, we set $U_t=\Omega_{[0, \,
2^{-2^t}]}$. By Corollary \ref{pol_appr} and (\ref{at}), there
exists a linear continuous operator $P^t:L_{q,v}(U_t) \rightarrow
{\cal P}_{r-1}(U_t)$ such that for any function $f\in
W^r_{p,g}(\Omega)$
\begin{align}
\label{bbb} \|f-P^tf\|_{L_{q,v}(U_t)} \underset{\mathfrak{Z}}
{\lesssim} 2^{-t\alpha}\rho(2^t).
\end{align}

{\bf Step 4.} Consider the cases $p=q$ and $p<q$, $\hat q\le 2$.
Denote $\hat U_t=U_t\backslash \cup_{s=0}^{t-1} G_s$, $t\in \N$.
We set $P_n^1f|_{\hat G_t}=P^1_{m_t^*,t,n}f|_{\hat G_t}$ for $0\le
t\le Nd$, $P_n^1f|_{\hat U_{Nd+1}}=P^{Nd+1}f|_{\hat U_{Nd+1}}$.
Denote $T^1_n=\left(\cup _{t=0}^{Nd} \hat
T^1_{m^*_t,t,n}\right)\cup \{\hat U_{Nd+1}\}$. Then $T^1_n$ is a
partition of $\Omega$ and
$$
{\rm card}\, T^1_n \stackrel{(\ref{lmt}),(\ref{ctmnt})}{\underset
{\mathfrak {Z}} {\lesssim}} \sum \limits_{0\le t\le Nd, \,
m^*_t=0} n\cdot 2^{-\varepsilon|t-t_*|}+ \sum \limits_{0\le t\le
Nd, \, m^*_t>0} 2^t\underset{\mathfrak{Z}, \varepsilon}{\lesssim}
n.
$$
In order to estimate Kolmogorov and linear widths, it is
sufficient to obtain an upper bound for
$\|f-P^1_nf\|_{L_{q,v}(\Omega)}$. In estimating Gelfand widths, we
obtain an upper bound of $\|f\|_{L_{q,v}(\Omega)}$ for $f\in
W^r_{p,g}(\Omega)$ such that $P^1_nf=0$.

We have
$$
\|f-P^1_nf\|^q_{L_{q,v}(\Omega)} =\sum \limits _{t=0}^{Nd} \|f-
P^1_{m^*_t,t,n}f\|^q_{L_{q,v}(\hat G_t)}+
\|f-P^{Nd+1}f\|^q_{L_{q,v}(\hat U_{Nd+1})}
\stackrel{(\ref{lmt}),(\ref{appr}),
(\ref{bbb})}{\underset{\mathfrak{Z}}{\lesssim}}
$$
$$
\lesssim \sum \limits_{t=0}^{Nd} 2^{-t\alpha q}\rho^q(2^t)
2^{-\frac{\delta q}{d}(Nd-t-\varepsilon|t-t_*|)}+n^{-\alpha q}
\rho^q(n)=:S^q
$$
(we used the fact that if $m_t^*>0$, then $Nd-t-\varepsilon
|t-t_*|<0$). Set $\varepsilon =\frac{d}{2\delta}\left|\alpha-
\frac{\delta}{d}\right|$. If $\alpha>\frac{\delta}{d}$, then we
put $t_*=0$ and get $S\underset{\mathfrak{Z}}{\lesssim}
n^{-\frac{\delta}{d}}$; if $\alpha<\frac{\delta}{d}$, then we take
$t_*=Nd$ and obtain $S\underset{\mathfrak{Z}}{\lesssim}
n^{-\alpha}\rho(n)$.

{\bf Step 5.} Consider the case $p<q$, $\hat q>2$. Denote $\hat
t_N=\left[\frac{\hat q}{2}Nd\right]$. Then for sufficiently large
$N$ we have $Nd+1< \hat t_N$.

For each $Nd+1\le t<\hat t_N$, $m\in \Z_+$ we construct the
partition $\hat T^2_{m,t,n}$ of $\hat G_t$.

{\bf Substep 5.1.} Let $\gamma=2^{1+\varepsilon}$,
$m_t=\lceil(1+\varepsilon)(t-Nd)\rceil$, $0\le m\le m_t$. Set
$$
j_{m,t}(s)=\left\lfloor \gamma^{t-Nd}2^{-m}s\right\rfloor,  \quad
\tau_{m,t}(s)=2^{-j_{m,t}(s)},\quad s\in \Z_+,
$$
$$
J_{m,t}=\{s\in \Z_+:\, j_{m,t}(s)\le 2^{t+1}, \, j_{m,t}(s+1)\ge
2^t\}.
$$
Then ${\rm card}\, J_{m,t}\underset{\mathfrak{Z}}{\lesssim} 2^mn
\cdot 2^{-\varepsilon(t -Nd)}$ for sufficiently small
$\varepsilon$.

Denote by $T^2_{m,t,n}$ the covering of $G_t$ by the sets $\Omega
_{[\tau_{m,t}(s+1),\tau_{m,t}(s)]}$, $s\in J_{m,t}$,
$\tau_{m,t}(s) \ne \tau_{m,t}(s+1)$. Then ${\rm card}\,
T^2_{m,t,n} \underset {\mathfrak{Z}} {\lesssim} 2^mn \cdot
2^{-\varepsilon(t -Nd)}$. If $\tau_{m,t}(s) \ne \tau_{m,t}(s+1)$,
then $\tau_{m,t}(s)- \tau_{m,t}(s+1) \asymp \tau_{m,t}(s)
\stackrel{(\ref{phi_z_z})}{\ge} \varphi(\tau_{m,t}(s))$.
Therefore, from the definition of $\tau_{m,t}(s)$ and from Lemma
\ref{kratn_pokr} it follows that for any set $U\in T^2_{m,t,n}$
\begin{align}
\label{l011} {\rm card}\, \{U'\in T^2_{m+l,t,n}:\, {\rm mes}\,
(U\cap U')>0\} \underset{\mathfrak{Z}}{\lesssim} 1, \;\; l=0, \,
1, \, -1.
\end{align}
By Corollary \ref{pol_appr} and (\ref{at}), for any set $E\in
T^2_{m,t,n}$ there exists a linear continuous operator
$P_E:L_{q,v}(E)\rightarrow {\cal P}_{r-1}(E)$ such that for any
function $f\in W^r_{p,g}(\Omega)$
\begin{align}
\label{peeee}
\|f-P_Ef\|_{L_{q,v}(E)}\underset{\mathfrak{Z}}{\lesssim}
2^{-t\alpha} \rho(2^t)\left\|\frac{\nabla^r
f}{g}\right\|_{L_p(E)}.
\end{align}

By Proposition \ref{pokr_razb}, there exist a partition $\hat
T^2_{m,t,n}$ of $\hat G_t$ and an injection ${\cal F}_{m,t,n}:
\hat T^2_{m,t,n} \rightarrow T^2_{m,t,n}$ such that
\begin{align}
\label{ct2} {\rm card}\, \hat T^2_{m,t,n}
\underset{\mathfrak{Z}}{\lesssim} 2^mn \cdot 2^{-\varepsilon(t
-Nd)}; \quad \text{ in particular,}\quad {\rm card}\, \hat
T^2_{m_t,t,n}\underset{\mathfrak{Z}}{\lesssim} 2^t;
\end{align}
for any $U\in \hat T^2_{m,t,n}$
\begin{align}
\label{2inters} {\rm card}\, \{U'\in \hat T^2_{m\pm 1,t,n}:\, {\rm
mes}\, (U\cap U')>0\} \underset{\mathfrak{Z}}{\lesssim} 1
\end{align}
and $U\subset {\cal F}_{m,t,n}(U)$. In addition, there exists a
linear continuous operator $P^2_{m,t,n}:L_{q,v}(\Omega)
\rightarrow {\cal S}_{r,\hat T^2_{m,t,n}}(\hat G_t)$ such that for
any function $f\in W^r_{p,g}(\Omega)$
\begin{align}
\label{2appr} \|f-P^2_{m,t,n}f\|_{p,q,\hat T^2_{m,t,n},
v}\underset{\mathfrak{Z}}{\lesssim} 2^{-t\alpha} \rho(2^t).
\end{align}
Notice that $\hat T^2_{m,t,n}$ can be defined as follows:
\begin{align}
\label{mtttt} \begin{array}{c} E \in \hat T^2_{m,t,n} \quad
\Leftrightarrow  \quad {\rm mes}\, E>0 \quad \text{and} \\ \exists
s\in J_{m,t} :\quad E=\hat G_t \cap \Omega
_{[\tau_{m,t}(s+1),\tau_{m,t}(s)]} \backslash \cup _{s'\in
J_{m,t},s'<s} \Omega _{[\tau_{m,t} (s'+1), \tau_{m,t}(s')]}; \\
\text{in this case,} \quad {\cal F}_{m,t,n}(E)=\Omega
_{[\tau_{m,t}(s+1),\tau_{m,t}(s)]}
\end{array}
\end{align}
(see the proof of Proposition \ref{pokr_razb}).

{\bf Substep 5.2.} Construct the partition $\hat T^2_{m,t,n}$ for
$m>m_t$. Observe that for any $s\in J_{m_t,t}$ the inequality
$|j_{m_t,t}(s+1)-j_{m_t,t}(s)|\le 1$ holds. Therefore, each
element of the covering $T^2_{m_t,t,n}$ coincides with
$\Omega_{[2^{-j},2^{-j+1}]}$ for some $j\in \{2^t, \, \dots , \,
2^{t+1}+1\}$. Recall that $\hat G_t \subset
G_t=\Omega_{[2^{-2^{t+1}}, \, 2^{-2^t}]}$. This together with
(\ref{mtttt}) implies that $$\hat T^2_{m_t,t,n}=\{D_{j-1}:\;
2^t+1\le j\le 2^{t+1}\}$$ and ${\cal
F}_{m_t,t,n}(D_{j-1})=\Omega_{[2^{-j},2^{-j+1}]}$.

Put
$$
\hat T^2_{m,t,n}=\{E\in {\bf R}_{j,m-m_t}:\; 2^t+1\le j\le
2^{t+1}\}.
$$
Then $\hat T^2_{m,t,n}$ is a partition of $\hat G_t$. From
(\ref{ceij}), (\ref{zzz}) and (\ref{ct2}) it follows that
\begin{align}
\label{ct22} {\rm card}\, \hat T^2_{m,t,n}\lesssim 2^{m-m_t}\cdot
2^t\asymp 2^mn\cdot 2^{-\varepsilon(t -Nd)}.
\end{align}
By (\ref{fx468jhdewml_8h}) and (\ref{ooo}), for any $E\in \hat
T^2_{m,t,n}$
\begin{align}
\label{2inters1} {\rm card}\, \{E'\in \hat T^2_{m\pm 1,t,n}:\,
{\rm mes}\, (E\cap E')>0\} \underset{\mathfrak{Z}}{\lesssim} 1.
\end{align}
For any $f\in L_{q,v}(\Omega)$ set $P^2_{m,t,n}f|_{D_{j-1}}=
\mathbb{P}_{j,m-m_t}f|_{D_{j-1}}$, $2^t+1\le j\le 2^{t+1}$,
$P^2_{m,t,n}f|_{\Omega \backslash \hat G_t}=0$. Then $P^2_{m,t,n}:
L_{q,v}(\Omega) \rightarrow {\cal S}_{r, \hat T^2_{m,t,n}}(\hat
G_t)$ is a linear continuous operator and for any $f\in
W^r_{p,g}(\Omega)$
$$
\|f-P^2_{m,t,n}f\|_{p,q,\hat T^2_{m,t,n},v} =\left( \sum \limits
_{j=2^t+1}^{2^{t+1}} \|f-\mathbb{P}_{j,m-m_t}f\|^p_{p,q, {\bf
R}_{j,m-m_t},v}\right)^{\frac 1p} \stackrel{(\ref{a1}),
(\ref{a2})}{\underset {\mathfrak{Z}}{\lesssim}}
$$
$$
\lesssim \left( \sum \limits _{j=2^t+1}^{2^{t+1}} j^{-p\alpha}
\rho^p(j)\cdot 2^{-\frac{ p\delta}{d}(m-m_t)}\left\|
\frac{\nabla^r f}{g}\right\|^p_{L_p(\Omega _{[2^{-j}, \,
2^{-j+1}]})} \right)^{\frac 1p} \underset{\mathfrak{Z}}{\lesssim}
$$
$$
\lesssim 2^{-t\alpha-(m-m_t)\frac{\delta}{d}}\rho(2^t)\left( \sum
\limits _{j=2^t+1}^{2^{t+1}} \left\| \frac{\nabla^r f}{g}
\right\|^p _{L_p(\Omega _{[2^{-j}, \, 2^{-j+1}]})}\right)^{\frac
1p} \stackrel{(\ref{nqt_z_1})}{\underset{\mathfrak{Z}}{\lesssim}}
2^{-t\alpha-(m-m_t)\frac{\delta}{d}}\rho(2^t).
$$
Thus, for any function $f\in W^r_{p,g}(\Omega)$
\begin{align}
\label{2appr1} \|f-P^2_{m,t,n}f\|_{p,q,\hat T^2_{m,t,n},
v}\underset{\mathfrak{Z}}{\lesssim}
2^{-t\alpha-(m-m_t)\frac{\delta}{d}} \rho(2^t).
\end{align}

{\bf Substep 5.3.} Let $f\in W^r_{p,g}(\Omega)$. Then
\begin{align}
\label{f_ptnf} \|f-P^{\hat t_N}f\|_{L_{q,v}(U_{\hat t_n})}
\stackrel{(\ref{bbb})}{\underset{\mathfrak{Z}}{\lesssim}}
n^{-\alpha \hat q/2}\rho(n^{\hat q/2}).
\end{align}

The further arguments are the same as in \cite{vas_sing}; here we
apply (\ref{custmtn}), (\ref{ctmnt}), (\ref{appr}), (\ref{ct2}),
(\ref{2inters}), (\ref{2appr}), (\ref{ct22}), (\ref{2inters1}),
(\ref{2appr1}) and (\ref{f_ptnf}).

{\it Proof of the lower estimate.} Observe that $\Omega$ contains
a cube $Q$ with the side of length $\lambda
\underset{\mathfrak{Z}}{\gtrsim} 1$, such that $g(x)
\underset{\mathfrak{Z}}{\gtrsim} 1$, $v(x)
\underset{\mathfrak{Z}}{\gtrsim} 1$ for any $x\in Q$. Therefore,
$$\vartheta _n(W^r_{p,g}(\Omega), \,
L_{q,v}(\Omega))\underset{\mathfrak{Z}} {\gtrsim} \vartheta
_n(W^r_p([0, \, 1]^d), \, L_q([0, \, 1]^d)).$$ Applying Theorem
\ref{sob_dn}, we get
$$
\vartheta_n(W^r_{p,g}(\Omega), \,
L_{q,v}(\Omega))\underset{\mathfrak{Z}} {\gtrsim}
n^{-\frac{\delta}{d}} \quad \mbox{for} \quad p=q\;\text{ or }\;
p<q, \;  \hat q\le 2,
$$
$$
\vartheta_n(W^r_{p,g}(\Omega), \,
L_{q,v}(\Omega))\underset{\mathfrak{Z}} {\gtrsim}
n^{-\min\{\theta_1, \, \theta_2\}} \quad \mbox{for} \quad p<q, \;
\hat q>2.
$$

Let $\psi\in C_0^\infty(0, \, 1)$, $\psi\ge 0$, $\int \limits _0^1
|\psi^{(r)}|^p\, dx=1$. Given $j\in \N$, we set $\psi_j(x', \,
x_d)=c_j\psi\left(2^jx_d-1\right)$; here $c_j>0$ is such that
$\left\|\frac{\nabla^r \psi_j}{g}\right\|_{L_p(\Omega)}=1$. Then
$${\rm supp}\, \psi_j\subset \{(x', \, x_d)\in \Omega:\,
2^{-j}<x_d<2^{-j+1}\}.$$ From (\ref{dist}) and order equalities
$g(x) \stackrel{(\ref{g0v0})}{\underset{\mathfrak{Z}}{\asymp}}
g(2^{-j})$, $v(x) \stackrel {(\ref{g0v0})}
{\underset{\mathfrak{Z}}{\asymp}} v(2^{-j})$,
$2^{-j}<x_d<2^{-j+1}$, it follows that
$c_j\underset{\mathfrak{Z}}{\asymp}
g(2^{-j})\varphi(2^{-j})^{-\frac{d-1}{p}}2^{\frac{j}{p}}\cdot
2^{-jr}$ and
$$\|\psi_j\|_{L_{q,v}(\Omega)} \underset{\mathfrak{Z}}{\asymp}
c_j\cdot v(2^{-j})\varphi(2^{-j})^{\frac{d-1}{q}} 2^{-\frac
jq}\underset{\mathfrak{Z}}{\asymp} j^{-\alpha}\rho(j)$$ (the last
relation follows from (\ref{g0v0}),  (\ref{phi_def}),
(\ref{limtrho}), (\ref{sigma_prop}), (\ref{alpha_prop}),
(\ref{rho_def})). If $2^t\le j<2^{t+1}$, then
$\|\psi_j\|_{L_{q,v}(\Omega)} \underset{\mathfrak{Z}}{\asymp}
2^{-t\alpha}\rho(2^t)$. Applying Lemma \ref{low}, we get
$$
\vartheta_n(W^r_{p,g}(\Omega), \,
L_{q,v}(\Omega))\underset{\mathfrak{Z}} {\gtrsim}
2^{-t\alpha}\rho(2^t) \vartheta_n(B_p^{2^t}, \, l_q^{2^t}).
$$

Let $t=Nd+1$. Then from (\ref{gluskin}), (\ref{gluskin_lin}) and
(\ref{gluskin_gelf}) it follows that
$$
\vartheta _n(W^r_{p,g}(\Omega), \,
L_{q,v}(\Omega))\underset{\mathfrak{Z}} {\gtrsim}
n^{-\alpha}\rho(n) d_n(B_p^{2n}, \, l_q^{2n})
\underset{p,q}{\asymp} \left\{\begin{array}{l} n^{-\alpha}\rho(n),
\quad p=q \; \text{ or }\; p<q, \; \hat q\le 2,
\\ n^{-\alpha-\min\{\frac 1p-\frac 1q, \, \frac 12-\frac {1}{\hat q}\}},
\quad \hat q>2.\end{array}\right.
$$

Let $t=\left[\frac{\hat q}{2}Nd\right]$, $\hat q>2$. Then
$$\vartheta_n(W^r_{p,g}(\Omega), \, L_{q,v}(\Omega))\underset{\mathfrak{Z}}
{\gtrsim} n^{-\hat q\alpha /2}\rho(n^{\hat q/2})
\vartheta_n(B_p^{[n^{\hat q/2}]}, \, l_q^{[n^{\hat q/2}]})
\underset{\mathfrak{Z}}{\asymp} n^{-\hat q\alpha /2}\rho(n^{\hat
q/2}).$$ This completes the proof.
\end{proof}
In conclusion, the author expresses her sincere gratitude to A.S.
Kochurov for reading the manuscript.

\end{document}